\providecommand{\ifcompress}{\iffalse}

\ifcompress\documentclass[10pt]{amsart}
\else
\documentclass[11pt,a4]{amsart}

\fi
\usepackage{amssymb}

\def\N{{\mathbb N}}

\def\be#1{\begin{equation}\label{#1}}
\def\N{{\mathbb N}}
\def\R{{\mathbb R}}

\def\bc{\begin{center}}
\def\ec{\end{center}}

\DeclareMathOperator{\supp}{supp}
\numberwithin{equation}{section}
\newenvironment{acknowledgements}{\subsection*{Acknowledgements}}{}
\newtheorem{definition}{Definition}[section]

\newtheorem{theorem}{Theorem}

\newtheorem{cor}[definition]{Corollary}
\newtheorem{lemma}[definition]{Lemma}
\newtheorem{proposition}[definition]{Proposition}
\newtheorem{example}[definition]{Example}
\newtheorem{remark}[definition]{Remark}

\newcommand{\half}{\mathchoice{\tfrac12}{\tfrac12}{\frac12}{\frac12}}
\newcommand{\sm}{\smallskip}

\newcommand{\Ps}[2][]{\P^{#1}\{#2\}}
\newcommand{\bPs}[2][]{\P^{#1}\bigl\{#2\bigr\}}

\newcommand{\Exp}{\mathbb{E}}		
\newcommand{\Eh}{\widehat{E}}		
\DeclareMathOperator{\Expo}{Exp}	
\newcommand{\cadlag}{c\`adl\`ag}
\renewcommand{\dh}{\hat{d}}
\newcommand{\rh}{\hat{r}}
\newcommand{\Th}{\hat{T}}
\newcommand{\smallxt}{\tilde{\smallx}}
\newcommand{\Tt}{\tilde{T}}
\newcommand{\Xt}{\tilde{X}}

\newcommand{\ee}{\end{equation}}

\newcommand{\eea}{\end{eqnarray}}
\newcommand{\bean}{\begin{eqnarray*}}
\newcommand{\eean}{\end{eqnarray*}}

\newif\ifpctex

\newcommand{\eps}{\varepsilon}

\newcommand{\E}{\mathbb{E}}
\newcommand{\Z}{\mathbb{Z}}
\renewcommand{\P}{\mathbb{P}}

\newcommand{\CL}{\mathcal{L}}
\newcommand{\law}{\CL}

\newcommand{\CM}{\mathcal{M}}

\newcommand{\Lip}{\mathrm{Lip}}
\newcommand{\lemref}[1]{Lemma~\ref{Lem:#1}}
\newcommand{\propref}[1]{Proposition~\ref{P:#1}}
\newcommand{\defref}[1]{Definition~\ref{Def:#1}}
\newcommand{\corref}[1]{Corollary~\ref{cor:#1}}

\newcommand{\bew}[1]{\begin{equation*}\label{#1}}
\newcommand{\bea}[1]{\begin{eqnarray}\label{#1}}
\newcommand{\beL}[2]{\begin{lemma}[#2]\label{#1}}
\newcommand{\beD}[2]{\begin{definition}[#2]\label{#1}}
\newcommand{\beT}[2]{\begin{theorem}[#2]\label{#1}}
\newcommand{\beP}[2]{\begin{proposition}[#2]\label{#1}}
\newcommand{\beC}[2]{\begin{cor}[#2]\label{#1}}

\newcommand{\D}{\displaystyle}

  \newcommand{\tno}{{_{\displaystyle\longrightarrow\atop n\to\infty}}}

\newcommand{\smallcal}[1]{
	{\mathchoice{\scriptstyle}{\scriptstyle}{\scriptscriptstyle}{\scriptscriptstyle}\mathcal{#1}}}
\newcommand{\smallx}{\smallcal{X}}
\newcommand{\smallt}{\smallcal{X}}

\newcommand{\tree}{\smallx}
\newcommand{\treen}{\tree_n}

\newcommand{\Bcl}{\overline{B}}
\newcommand{\bdot}{\boldsymbol{\cdot}}

\usepackage{color}
\definecolor{nb}{rgb}{.6,.176,1}
\definecolor{sienna}{rgb}{.92,.222,.176}
\definecolor{darkgreen}{rgb}{0,.5,0}

\newcommand{\nbd}{\protect\nobreakdash-\hspace{0pt}}	
\newcommand{\Rtree}{$\R$\nbd tree}

\newcommand{\ton}[1][n]{\,\displaystyle\mathop{\longrightarrow}^{#1\to\infty}\,}	
\newcommand{\Tno}[1][n]{\,\displaystyle\mathop{\Longrightarrow}_{#1\to\infty}\,}	
\newcommand{\convlaw}[1][n]{\,\displaystyle\mathop{\Longrightarrow}_{#1\to\infty}^{\CL}\,}	
\newcommand{\eqlaw}{\overset{\CL}{=}}

\renewcommand{\E}{\mathcal{E}}
\renewcommand{\D}{\mathcal{D}}
\newcommand{\T}{\mathcal{T}}

\newcommand{\A}{\mathcal{A}}
\newcommand{\cont}{\mathcal{C}}

\newcommand{\indicator}[1]{\mathbf{1}_{#1}}
\newcommand{\restricted}[1]{{\restriction}_{#1}}

\DeclareMathOperator{\diam}{diam}
\newcommand{\nlim}{\lim_{n\to\infty}}
\newcommand{\nlimsup}{\limsup_{n\to\infty}}
\newcommand{\nliminf}{\liminf_{n\to\infty}}
\newcommand{\diffd}{\mathrm{d}}
\newcommand{\integralspace}{\/\mathchoice{\;}{\>}{\,}{}}		
\newcommand{\integral}[4]{\int_{#1}^{#2} #3 \integralspace\diffd#4}	
\newcommand{\plainint}[2]{\int #1 \integralspace\diffd#2}	
\newcommand{\inta}[3]{\integral{#1}{}{#2}{#3}}				

\newcommand{\wtspace}{\mathchoice{\,}{}{}{}}			
\newcommand{\set}[2]{\{\wtspace #1 : #2 \wtspace\}}		
\newcommand{\bset}[2]{\bigl\{\mspace{1mu}#1 \wtspace:\wtspace #2 \mspace{1mu}\bigr\}}

\newcommand{\bigldelimiter}[1]{\mathchoice{\bigl#1}{\bigl#1}{{\textstyle#1}}{{\scriptstyle#1}}}
\newcommand{\bigrdelimiter}[1]{\mathchoice{\bigr#1}{\bigr#1}{{\textstyle#1}}{{\scriptstyle#1}}}
\renewcommand{\(}{\bigldelimiter(}
\renewcommand{\)}{\bigrdelimiter)}

\newcommand{\TLambda}{{{\mathcal T}}_{\mbox{\tiny$\Lambda$}}}
\newcommand{\TLambdacl}{\bar{\TLambda}}
\newcommand{\BrTLambda}{\mathrm{Br}(\TLambdacl)}

\newcommand{\epsh}{\hat\eps}
\newcommand{\Q}{\mathbb{Q}}
\newcommand{\clproperty}{closed-interval property}

\newcommand{\Pn}[2]{P_{n,#2}^{#1}}

\newcommand{\proofstep}[1]{\par\smallskip\emph{#1}\,}

\ifcompress
\renewcommand{\sm}{}
\usepackage{a4wide}
\fi

\author{Siva Athreya}
\address{Siva Athreya \\ Indian Statistical Institute
8th Mile Mysore Road  \\ Bangalore 560059, India.}
\email{athreya@isibang.ac.in}
\thanks{}

\author{Wolfgang L\"ohr}
\address{Wolfgang L\"ohr\\ Fakult\"at f\"ur Mathematik\\ Universit\"at Duisburg-Essen\\
Thea-Leymann-Str.\ 9\\ 45117 Essen, Germany}
\thanks{}
\email{wolfgang.loehr@uni-due.de}

\author{Anita Winter}
\address{Anita Winter\\ Fakult\"at f\"ur Mathematik\\ Universit\"at Duisburg-Essen\\
Thea-Leymann-Str.\ 9\\ 45117 Essen, Germany}
\thanks{}
\email{anita.winter@uni-due.de}


\begin{document}

\title[Invariance principle on trees]{\vspace*{-3.25\baselineskip}Invariance principle for variable speed random walks on
trees}
\date{\today.\hspace{1em} \emph{Preprint of}: {Ann.\ Probab.\ 45(2):625--667, 2017}}

\begin{abstract}
We consider stochastic processes on complete, locally compact tree-like metric spaces $(T,r)$ on their ``natural scale'' with boundedly finite speed measure $\nu$. Given a triple $(T,r,\nu)$
such a speed-$\nu$ motion on $(T,r)$  can be characterized as the unique strong Markov process which  if restricted to compact subtrees satisfies for all
$x,y\in T$ and all positive, bounded measurable $f$,
\begin{equation}
\label{e:abstract}
\mathbb{E}^x\Bigl[\int^{\tau_y}_0\mathrm{d}s\,f(X_s)\Bigr]=2\int_T\nu(\mathrm{d}z)\,r\(y,c(x,y,z)\)f(z)<\infty,
\end{equation}
where  $c(x,y,z)$ denotes the branch point generated by $x,y,z$. If  $(T,r)$ is a discrete tree, $X$ is a
continuous time nearest neighbor random walk which jumps from $v$ to $v'\sim v$ at rate
$\tfrac{1}{2}\cdot\(\nu(\{v\})\cdot r(v,v')\)^{-1}$. If $(T,r)$ is path-connected, $X$ has continuous paths and equals the $\nu$-Brownian motion which was recently constructed in \cite{AthreyaEckhoffWinter2013}. In this paper we show that speed-$\nu_n$ motions on $(T_n,r_n)$ converge weakly in path space to the speed-$\nu$ motion on $(T,r)$
provided that the underlying triples of metric measure spaces converge in the Gromov-Hausdorff-vague topology introduced in \cite{ALW2}.
\end{abstract}

\maketitle

\vspace*{-2\baselineskip}

{\tiny
\begin{quote}
\tableofcontents
\end{quote}
}

\vspace{-2\baselineskip}


\section{Introduction and main result (Theorem~\ref{T:001})}
\label{S:intro}
Fifty years ago in \cite{Stone1963}  Markov processes were considered which have in common that their state
spaces are closed subsets of the real line and that their random trajectories {\em ``do not jump over points''}.
When put in their {\em ``natural scale''} these processes are determined by their {\em ``speed measure''}. Stone
argues that in some sense the processes depend continuously on the speed measures.  The most classical example
is the symmetric simple random walk on $\mathbb{Z}$ which, after a suitable rescaling, converges to standard
Brownian motion.
If you rescale edge lengths by a factor $\tfrac{1}{\sqrt{n}}$ and speed up time by a factor $n$, then you might think of the rescaled random walk as such a process with speed measure
$\tfrac{1}{\sqrt{n}}q(\sqrt{n}\,\boldsymbol{\cdot})$, where $q$ denotes the counting measure on $\mathbb{Z}$,
and of the standard Brownian motion as such a process whose speed measure equals the Lebesgue measure on $\mathbb{R}$.

In the present paper we want to extend this result from $\R$-valued Markov processes to Markov processes which take values in tree-like metric spaces. Before we state our main result precisely, we do the preliminary work and define the space of rooted metric boundedly finite measure trees equipped with pointed Gromov-vague topology and
give our notion of convergence in path space.

\begin{definition}[Rooted metric boundedly finite measure trees]\hspace{1cm}
\begin{enumerate}
\item A \emph{pointed Heine-Borel space} $(X,r,\rho)$ consists of a Heine-Borel space\footnote{Recall that a
	\emph{Heine-Borel space} is a metric space in which every bounded closed subset is compact. Note that
	every Heine-Borel space is complete, separable and locally compact.} $(X,r)$ and a distinguished point
	$\rho\in X$.
\item A \emph{rooted metric tree}  is a pointed Heine-Borel space $(T,r,\rho)$, which is both
	\emph{$0$-hyperbolic}, or equivalently, satisfies the \emph{four point condition}, i.e.,
	\begin{equation} \label{e:4point}
	\begin{aligned}
		&r(x_1,x_2)+r(x_3,x_4) \\
		  &\le \max\bigl\{r(x_1,x_3)+r(x_2,x_4),\,r(x_1,x_4)+r(x_2,x_3)\bigr\},
	\end{aligned}
	\end{equation}
	holds for all $x_1,x_2,x_3,x_4\in T$, and \emph{fine}, i.e., for all $x_1,x_2,x_3\in T$ there is a
	(necessarily unique) point $c(x_1,x_2,x_3)\in T$, such that for $i,j\in\{1,2,3\}$, $i\ne j$,
		\begin{equation}
		\label{e:br}
			r\big(x_i,c(x_1,x_2,x_3)\big)+r\big(x_j,c(x_1,x_2,x_3)\big)=r(x_i,x_j).
		\end{equation}
	The point $c(x_1,x_2,x_3)$ is referred to as \emph{branch point}, and the distinguished point
	$\rho\in T$ as the \emph{root}.
\item In a rooted metric tree $(T, r, \rho)$ we define for $a,b\in T$ the intervals
	\begin{equation} \label{e:interval}
	   [a,b]:=\bset{x\in T}{r(a,x)+r(x,b)=r(a,b)},
	\end{equation}
	$(a,b):=[a,b]\setminus\{a,b\}$, $[a,b):=[a,b]\setminus\{b\}$ and $(a,b]:=[a,b]\setminus\{a\}$. We say
	that $x,y \in T$ are connected by an \emph{edge}, in symbols $x\sim_T y$ or simply $x\sim y$, iff
	\begin{equation}\label{e:edge}
		x\ne y \qquad\text{and}\qquad [x,y] = \{x,y\}.
	\end{equation}
	If $x\sim y$ and $x\in[\rho,y]$, we call the pair $(x,y)$ an \emph{oriented edge} of length $r(x,y)$.
\item A \emph{rooted metric boundedly finite measure tree} $(T,r,\rho,\nu)$ consists of a rooted metric tree
	$(T,r,\rho)$ and a measure $\nu$ on $(T,{\mathcal B}(T))$ which is finite on bounded sets and has full
	support, $\supp(\nu)=T$.
\end{enumerate}
\label{Def:002}
\end{definition}\sm

\begin{remark}[$\R$-trees versus trees with edges]
	A metric tree is connected (i.e.\ is an \Rtree) if and only if it has no edges. Due to separability,
	there can be only countably many edges.
\label{Rem:001}
\hfill$\qed$
\end{remark}\sm

We will establish a one-to-one correspondence between rooted metric boundedly finite measure trees $(T,r,\rho,\nu)$
and strong Markov processes $X=(X_t)_{t\ge 0}$ with values in $(T,r)$ starting at $\rho$.  When $(T,r)$ is compact such a process can be characterized by the occupation time formula given in (\ref{e:abstract}) (see Proposition~\ref{P:001}).
For general rooted metric boundedly finite measure trees the corresponding Markov process is associated with a regular Dirichlet form (see Definition~\ref{Def:001}).
We will refer to this Markov process as {\em speed-$\nu$ motion on $(T,r)$} or {\em variable speed motion associated to $\nu$ on $(T,r)$}.
If $(T,r)$ is path-connected, then $X$ has continuous paths and equals the so-called $\nu$\nbd Brownian motion
on $(T,r)$, which was recently constructed in \cite{AthreyaEckhoffWinter2013}.
On the other hand, if $(T,r)$ is discrete, $X$ is a continuous time nearest neighbor Markov chain which jumps
from $v$ to $v'\sim v$ at rate
\begin{equation}
\label{e:rates}
   \gamma_{vv'}:=\tfrac{1}{2}\cdot\(\nu(\{v\})\cdot r(v,v')\)^{-1}
\end{equation}
(see Lemma~\ref{L:009}).

The invariance principle which we are going to state says that a sequence of variable speed motions converges in path space to a limiting variable speed motion whenever the underlying
metric measure trees converge in the pointed Gromov-Hausdorff-vague topology which was recently introduced in \cite{ALW2}.
In particular, it was shown that convergence in pointed Gromov-Hausdorff-vague topology is equivalent to convergence in pointed Gromov-vague topology together with the {\em uniform local
lower mass-bound property}, i.e., for each $\delta,R>0$,
\begin{equation}
\label{e:assume2}
	   \liminf_{n\to \infty}\inf_{x\in B_n(\rho_n,R)}\nu_n\bigl(B_n(x,\delta)\bigr)>0
\end{equation}
(see Proposition~\ref{P:013}).
Here, $B_n(x,R)=\bset{y\in T_n}{r_n(x,y)<R}$ is the ball around $x$ with radius $R$ in the metric space $(T_n,r_n)$.
In the introduction we recall only the definition of Gromov-vague topology.
For a more elaborate discussion of the topology, we refer the reader to Section~\ref{S:topology}.

We call two rooted metric  measure trees $(T,r,\rho,\nu)$ and $(T',r',\rho',\nu')$ equivalent iff there is an isometry $\varphi$ between $(T,r)$ and $(T',r')$
such that $\varphi(\rho)=\rho'$ and $\nu\circ\varphi^{-1}=\nu'$. Denote
\begin{equation}
\label{e:mathbbT1}
\begin{aligned}
	&\mathbb{T}
  := \big\{\text{equivalence classes of rooted metric boundedly finite measure trees}\big\}.
\end{aligned}
\end{equation}
Let $\tree:=(T,r,\rho,\nu)$, $\tree_1:=(T_1,r_1,\rho_1,\nu),\,\tree_2:=(T_2,r_2,\rho_2,\nu), \ldots$ be
in $\mathbb{T}$. We say that $(\tree_n)_{n\in\mathbb{N}}$ converges to $\tree$ in {\em pointed Gromov-vague topology} iff
there are a pointed metric space $(E,d_E,\rho_E)$ and isometries $\varphi_n\colon T_n\to E$ with
$\varphi_n(\rho_n)=\rho_E$, for all $n\in\mathbb{N}$, as well as an isometry $\varphi\colon T\to E$
with $\varphi(\rho)=\rho_E$ such that the sequence of image measures
$(\varphi_{n\ast}\nu_n)\restricted{B(\rho_E,R)}$ restricted to the ball of radius $R$ around the
root converges weakly for all but countably many $R>0$.

Before we are in a position to state our main scaling result, notice that the approximating Markov processes may live on different spaces. We therefore agree on the following:

\begin{definition}[A notion of convergence in path space]
For every $n\in\N \cup \{\infty\}$, let $X^n$ be a c\`adl\`ag process with values in a metric space $(T_n,r_n)$.
\begin{enumerate}
\item\label{pathspace} We say that $(X^n)_{n\in\N}$ converges to $X^\infty$ weakly in path space (resp.\
	f.d.d.) if there exists a metric space $(E,d_E)$ and isometric embeddings $\phi_n\colon T_n \to E$,
	$n\in\N\cup\{\infty\}$, such that $(\phi_n\circ X^n)_{n\in\N}$ converges to $\phi_\infty \circ X^\infty$
	weakly in Skorohod path space (resp.\ f.d.d.).
\item We say that $(X^n)_{n\in\N}$ converges to $X^\infty$ in the one-point compactification weakly in path-space
	(resp.\ f.d.d.) if there exists a locally compact space $(E,d_E)$ and embeddings as in \ref{pathspace}
	such that we have weak path-space (resp.\ f.d.d.) convergence in the one-point compactification
	$E\cup\{\infty\}$, where the processes are defined to take the value $\infty$ after their lifetimes.
\end{enumerate}
\label{Def:003}
\end{definition}\sm

To be in a position to state our invariance principle, we recall the notion of the one-point compactification
$\Eh:=E\cup\{\infty\}$ of a separable, locally compact (but non-compact) metric space $E$, and the life time
$\zeta$ of a $E$-valued strong Markov process, i.e.,
\begin{equation}
\label{e:zeta}
   \zeta:=\inf\big\{t\ge 0:\,X_t=\infty\big\}.
\end{equation}

Our main result is the following:
\begin{theorem}[Invariance principle]
Let\/ $\tree:=(T,r,\rho,\nu)$,  $\tree_1:=(T_1,r_1,\rho_1,\nu_1)$, $\tree_2:=(T_2,r_2,\rho_2,\nu_2), \ldots$ be
in\/ $\mathbb{T}$.
Let\/ $X$ be the speed-$\nu$ motion on $(T,r)$ starting in $\rho$, and for all $n\in\mathbb{N}$,
let\/ $X^n$ be the speed-$\nu_n$ motion on\/ $(T_n,r_n)$ started in\/ $\rho_n$.
Assume that the following conditions hold:
\begin{itemize}
	\item[(A0)] For all\/ $R>0$,
	\begin{equation}
	\label{e:120}
		\limsup_{n\to\infty}\sup\bset{r_n(x,z)}{x\in B_n(\rho_n,R),\, z\in T_n,\, x\sim z}<\infty.
	\end{equation}
	\item[(A1)] The sequence\/ $(\tree_n)_{n\in\mathbb{N}}$ converges to\/ $\tree$
		pointed Gromov-vaguely.
	\item[(A2)] The uniform local lower mass-bound property (\ref{e:assume2}) holds.
\end{itemize}
Then the following hold:
\begin{enumerate}
	\item\label{i:pathspace} $X^n$ converges in the one-point compactification weakly in path-space to a
		process\/ $Y$, such that\/ $Y$ stopped at infinity has the same distribution as the speed-$\nu$
		motion\/ $X$.
		In particular, if\/ $X$ is conservative (i.e.\ does not hit infinity), then\/ $X^n$ converges
		weakly in path-space to\/ $X$.
	\item\label{i:fdd} If\/ $\sup_{n\in\N} \diam (T_n, r_n) < \infty$, where $\diam$ is the diameter,
		and we assume (A1) but not (A2), then\/ $X^n$ converges f.d.d.\ to\/ $X$.
\end{enumerate}
\label{T:001}
\end{theorem}\sm

\begin{remark}[Entrance law]
Let\/ $\tree:=(T,r,\rho,\nu)$,  $\tree_1:=(T_1,r_1,\rho_1, \nu_1)$, $\tree_2:=(T_2,r_2,\rho_2,\nu_2), \ldots$  in
$\mathbb{T}$ be such that $\tree_n\tno\tree$ Gromov-Hausdorff-vaguely.
The statement of Theorem~\ref{T:001}(i) reflects the fact that it is possible that
the approximating speed-$\nu_n$ motions on $(T_n,r_n)$, as well as their limit processes on the one-point
compactification, are recurrent but the speed\nbd $\nu$ motion on $(T,r)$ is not. Note that in such a situation we
obtain an entrance law and that the limit processes cannot be a strong Markov processes.
We explain this in detail in Example~\ref{ex:entrance}.
\label{Rem:005}
\hfill$\qed$
\end{remark}\sm

We want to briefly illustrate this invariance principle with a first non-trivial example which was established in \cite{Croydon2008}. Further examples and the relation of Theorem~\ref{T:001}
to the existing literature are discussed in Section~\ref{S:example}.
\begin{example}[RWs on GW-trees converge to BM on the CRT] \label{example1} Consider a Galton-Watson process in discrete time
whose offspring distribution is critical and has finite (positive) variance $\sigma^2$. For each $n\in\mathbb{N}$, let
${\mathcal T}_n$ be the corresponding GW-tree conditioned on having $n$ vertices.
Given ${\mathcal T}_n$,
whenever $v'\sim_{{\mathcal T}_n} v$, put $r_n(v,v'):=\tfrac{\sigma}{\sqrt{n}}$, and let
$\nu_n(\{v\}):=\tfrac{\deg(v)}{2n}$ for all $v\in{\mathcal T}_n$, where $\deg$ denotes the degree of node.
Notice that given ${\mathcal T}_n$, the speed-$\nu_n$ random walk on $({\mathcal T}_n,r_n)$ is the symmetric
nearest neighbor random walk on ${\mathcal T}_n$ with edge lengths rescaled by a factor $\tfrac{\sigma}{\sqrt{n}}$ and with exponential jump rates
\begin{equation}
\label{e:gammaGW}
   \gamma_n(v)=\tfrac{1}{2\nu_n(\{v\})}\sum_{v'\sim v}r^{-1}_n(v,v')=\tfrac{1}{2}\cdot
   \tfrac{2n}{\deg(v)}\cdot \deg(v) \tfrac{\sqrt{n}}{\sigma}=\sigma^{-1}\cdot n^{\frac{3}{2}}.
\end{equation}

Denote by $\mu_n^{\mathrm{ske}}$ the normalized length-measure (see Section~\ref{Sub:setup}) on the
path-connected tree $\overline{\T}_n$ spanned by ${\mathcal T}_n$.
Then it is known that $(\overline{\mathcal T}_n,r_n,\mu^{\mathrm{ske}}_n)$ converges Gromov-vaguely in
distribution to some random, compact, path-connected metric measure tree $({\mathcal T},r,\mu)$, where
$({\mathcal T},r)$ is the so-called Brownian continuum random tree (or shortly, the CRT), and $\mu$ the
``leaf-measure'' (see, for example, \cite[Theorem~23]{Aldous1993}).
As the Prohorov distance between $\nu_n$ and $\mu_n^{\mathrm{ske}}$ is not greater than $\tfrac{\sigma}{2\sqrt{n}}$,
$({\mathcal T}_n,r_n,\nu_n)$ also converges  Gromov-vaguely to $({\mathcal T},r,\mu)$ by \cite[Lemma~2.10]{ALW2}.
Furthermore it
 is known that the family $\{\nu_n;\,n\in\mathbb{N}\}$ satisfies the uniform local lower mass-bound property (compare \cite[Corollary~19]{Aldous1993} together with Proposition~\ref{P:013}).

We can therefore conclude from  Theorem~\ref{T:001} that given a realization of a sequence $({\mathcal T}_n)_{n\in\mathbb{N}}$ converging Gromov-weakly to some ${\mathcal T}$, the symmetric
random walk with jumps rescaled by $\tfrac{1}{\sqrt{n}}$ and time speeded up by a factor of $n^{\frac{3}{2}}$
converges to $\mu$-Brownian motion on the CRT. This was first conjectured in \cite[Section~5.1]{MR93f:60010} and proved in \cite{Croydon2008}.
A more general result on homogeneous scaling limits of random walks on graph trees towards diffusions on continuum trees was established in \cite{Croydon2010}.
We will discuss in Section~\ref{Sub:Croydon} how this result is covered by our invariance principle.
\label{Exp:002}
\hfill$\qed$
\end{example}\sm


For the proof of the invariance principle we use the following approach. We first use techniques from Dirichlet
forms to construct the speed-$\nu$ motion on $(T,r)$. We continue showing tightness based on a version of Aldous' stopping time criterion
(Proposition~\ref{P:Aldous}), and then identify the limit. As we are working with Dirichlet forms, one might be tempted to  show f.d.d.-convergence of the motions by
verifying the Mosco-convergence introduced in \cite{Mosco69}  (compare also \cite{Mosco94} for its application to Dirichlet
forms). It turns out, however, that this is tedious, and we rather identify the limit via the occupation time formula (\ref{e:abstract}).
For that, we first restrict ourselves to limit metric (finite) measure trees which are compact, and show that
any limit point must be a strong Markov process satisfying \eqref{e:abstract}.
We then reduce the general case to the case of compact limit trees by showing that there are suitably many hitting times which converge.\sm





The rest of the paper is organized as follows: In Section~\ref{S:Dirichlet} we construct the speed-$\nu$ motion on $(T,r)$ and present occupation time formula (\ref{e:abstract}).
In Section~\ref{S:topology} we introduce all the topological concepts needed to deal with convergence of the underlying metric measure spaces. In Section~\ref{S:tight} we prove the tightness of a sequence of speed-$\nu_n$ motions on $(T_n,r_n)$ provided that the underlying spaces $(T_n,r_n,\nu_n)_{n\in\mathbb{N}}$ converges. In Section~\ref{S:ident}
we show that any limit point satisfies the strong Markov property and that its occupation time formula agrees
with that of the limit variable speed motion.
In Section~\ref{S:proofT001} we collect all the ingredients to present the proof of Theorem~\ref{T:001}. Finally, in Section~\ref{S:example} we present examples and relate our result to the existing literature.


\section{The speed-$\nu$ motion on $(T,r)$ and its Dirichlet form}
\label{S:Dirichlet}
In this section we will use Dirichlet form techniques to construct the variable speed motions. We will follow the lines of  \cite{AthreyaEckhoffWinter2013} where the variable speed
motion was constructed on path-connected rooted metric measure trees, or  rooted measure {\em $\R$-trees} for short.
The main idea behind the generalization to arbitrary  rooted metric measure trees is the  presentation of a universal notion of the length measure and the gradient. This will be given in Subsection~\ref{Sub:setup}. In Subsection~\ref{Sub:form} we associate the variable speed motion with a Dirichlet form and establish in
Subsection~\ref{Sub:occupation} the occupation time formula.
We will revise
(where necessary)
the proofs given in   \cite{AthreyaEckhoffWinter2013} to the larger class of underlying  rooted metric measure trees.

\subsection{The set-up}
\label{Sub:setup}
In this subsection we discuss preliminaries that are required to construct the variable speed motions.

Recall rooted metric trees and rooted $\R$-trees from Definition~\ref{Def:002}, and
notice that a rooted metric tree $(T,r,\rho)$ can be embedded isometrically into an \emph{\Rtree}, i.e.\ a
path-connected rooted metric tree (see, for example, Theorem~3.38 in \cite{Evans2008}).
Furthermore, there is a unique (up to isometry) smallest rooted $\mathbb{R}$-tree, $(\bar{T},\bar{r},\rho)$, which contains $(T,r,\rho)$ (compare, e.g., \cite[Remark~2.7]{LoehrVoisinWinter2013}).  $(\bar{T},\bar{r})$ is the smallest $\mathbb{R}$-tree in
the following sense:
if $(\Th,\rh)$ is another \Rtree\ with  $T\subseteq \Th$, and $\rh$ extends $r$, then there is a unique isometric
embedding $\phi\colon \bar{T} \to \Th$ such that $\phi\restricted{T}$ is the identity on $T$. Heuristically,
$(\bar{T}, \bar{r})$ is obtained from $(T,r)$ by replacing edges with line segments of the appropriate length.

Given a rooted metric tree $(T,r,\rho)$, we can define a partial order (with respect to $\rho$), $\le_\rho$, on
$T$ by saying that $x\le_\rho y$ for all $x,y\in T$ with $x\in[\rho,y]$.

To be in a position to capture that our variable speed motions are processes on ``natural scale'' we need the notion of a length measure. For $\R$-trees it was first introduced in \cite{EvaPitWin2006}.  It turns out that this measure can be constructed on any separable $0$-hyperbolic metric space provided that we have fixed a reference point, say the root $\rho$.
Let therefore $(T,r,\rho)$ be a rooted metric tree, and  $\mathcal B(T)$  the Borel-$\sigma$-algebra of $(T,r)$.
We denote the set of {\em isolated points} (other than the root)  by $\mathrm{Iso}(T,r,\rho)$, and define the {\em skeleton} of $(T,r,\rho)$ as
\begin{equation}
\label{sce}
   {T}^o:=\mathrm{Iso}(T,r,\rho)\cup\bigcup\nolimits_{a\in {T}}\,(\rho,a).
\end{equation}

Recall that  rooted metric trees are Heine-Borel spaces and thus separable, and
observe that if
${T}^\prime \subset {T}$ is a
dense countable set, then (\ref{sce}) holds with ${T}$ replaced by
${T}^\prime$. In particular, ${T}^o \in {\mathcal B}({T})$ and
${\mathcal B}({T})\restricted{{T}^o}=\sigma(\{(a,b);\,a,b\in {T}^\prime\})$, where
${\mathcal B}({T})\restricted{{T}^o}:=\{A \cap {T}^o;\,A\in{\mathcal B}({T})\}$.
Hence, there exist a unique $\sigma$-finite measure $\lambda^{(T,r,\rho)}$ on $T$,
such that $\lambda^{(T,r,\rho)}({T}\setminus {T}^o)=0$ and for all $a\in T$,
\begin{equation}
\label{length}
   \lambda^{(T,r,\rho)}\((\rho,a]\)=r(\rho,a).
\end{equation}

\begin{definition}[Length measure] Let $(T,r,\rho)$ be a rooted metric tree.
The unique  $\sigma$-finite measure $\lambda^{(T,r,\rho)}$ satisfying (\ref{length}) and\/ $\lambda^{(T,r,\rho)}({T}\setminus {T}^o)=0$
is called the {\em length measure} of\/ $(T,r,\rho)$.
\label{Def:007}
\end{definition}\sm

\begin{remark}[Length measure; particular instances] \pushQED{\qed} \hspace{.3cm}
\begin{enumerate}
\item
If\/ $(T,r)$ is an $\R$-tree, then $\lambda^{(T,r,\rho)}$  does not depend on the root $\rho$, and
is the trace onto $T^o$ of the $1$-dimensional Hausdorff-measure on~$T$.
\item If\/ $(T,r)$ is discrete as a topological space, i.e.\ all points in $T$ are isolated, the length measure
shifts all the ``length'' sitting on an edge to the end point which is further away from the root.
In this case it does explicitly depend on the root.
\item\label{it:pi} In general, let\/ $(\bar T,\bar r)$ be the \Rtree\ spanned by $(T,r)$ and $\pi\colon \bar T \to T$
defined by
\begin{equation}
\label{e:pilength}
   \pi(x)
 :=
    \inf\big\{y\in T:\,x\le_\rho y\big\},
\end{equation}
for all $x\in \bar{T}$.
Note that $\pi$ is well defined because $T$ is closed and satisfies \eqref{e:br}. It is therefore easy to check that
\begin{equation}
\label{e:lengthpi}
   \lambda^{(T,r,\rho)} = \pi_\ast\lambda^{(\bar{T}, \bar{r})}. \qedhere
\end{equation}
\end{enumerate}
\label{Rem:020}
\end{remark}\sm

In order to characterize the variable speed motion analytically (via Dirichlet forms), we use a concept of weak
differentiability.
Denote the space of continuous functions $f\colon T\to\mathbb{R}$ by ${\mathcal C}(T)$.
We call a function $f\in {\mathcal C}(T)$ {\em locally absolutely continuous} if and only if for all
$\varepsilon>0$ and all subsets $S\subseteq T$ with $\lambda^{(T,r,\rho)}(S)<\infty$ there exists a
$\delta=\delta(\varepsilon,S)$ such that if $[x_1,y_1],\ldots,[x_{n},y_n]\subseteq S$ are disjoint arcs with
$\sum_{i=1}^{n} r(x_i,y_i)<\delta$ then $\sum_{i=1}^{n}\bigl|f(x_i)-f(y_i)\bigr|<\varepsilon$.
Put
\begin{equation}\label{mathcalA}
   {\mathcal A}={\mathcal A}^{(T,r)}
 :=
   \big\{f\in{\mathcal C}(T):\,f\mbox{ is locally absolutely continuous}\big\}.
\end{equation}
Of course, if $(T,r)$ is discrete, then ${\mathcal A}$ equals the space ${\mathcal C}(T)$ of continuous functions.

The definition of the gradient is then based on the following observation which was proved for $\R$-trees in
\cite[Proposition~1.1]{AthreyaEckhoffWinter2013}.
\begin{proposition}[Gradient]
	Let $f\in\mathcal A$. \label{P:grad}
	There exists a unique (up to $\lambda=\lambda^{(T,r,\rho)}$-zero sets) function
	$g\in L_{\mathrm{loc}}^1(\lambda^{(T,r,\rho)})$ such that
	\begin{equation}\label{con.1}
	   f(y)-f(x)
	 =
	    \int_{[\rho,y]}\lambda(\mathrm{d}z)\,g(z)-\int_{[\rho,x]}\lambda(\mathrm{d}z)\,g(z),
	\end{equation}
	for all $x,y\in T$.
	Moreover, $g$ is already uniquely determined  (up to $\lambda^{(T,r,\rho)}$-zero sets) if we only
	require (\ref{con.1}) to hold for all $x\le_\rho y$.
\end{proposition}\sm

\begin{proof}
	For $f\in\A$, we define the linear extension $\bar{f}\colon \bar{T} \to \R$ by
	$\bar{f}\restricted{T}:=f$ and
	\begin{equation}
		\bar{f}(v) := \tfrac{r(v,y)}{r(x,y)} f(x) + \tfrac{r(v,x)}{r(x,y)}f(y),
	\end{equation}
	whenever $(x,y)$ is an edge of $T$ and $v\in[x,y]\subseteq \bar{T}$.
	By \cite[Proposition~1.1]{AthreyaEckhoffWinter2013}, there is $\bar{g}\colon \bar{T} \to \R$ such that
	\eqref{con.1} holds for $x,y\in\bar{T}$ and $\bar{\lambda}:=\lambda^{(\bar{T},\bar{r})}$ instead of
	$\lambda$.
	It is easy to see from the definition of $\bar{f}$ that $\bar{g}$ is constant on edges of $T$ and hence
	$g\colon T\to \R$ is well defined by $g\circ \pi :=\bar{g}$, with $\pi$ defined in
	Remark~\ref{Rem:020}\ref{it:pi}. By (\ref{e:lengthpi}),
	\begin{equation}
    \label{e:proofgrad}
    \begin{aligned}
		f(y)-f(x)
   &= \int\nolimits_{[\rho, y]}\mathrm{d}\bar{\lambda}\,{\bar{g}} - \int\nolimits_{[\rho, x]}\mathrm{d}\bar{\lambda}\,{\bar{g}}
   \\
   &=\int\nolimits_{[\rho, y]}\mathrm{d}\lambda\,{g} - \int\nolimits_{[\rho, x]}\mathrm{d}\lambda\,{g}.
\end{aligned}
\end{equation}

Uniqueness and integrability of $g$ follow from the corresponding properties of $\bar{g}$.
\end{proof}\sm

The statement of Proposition~\ref{P:grad} yields a general notion of a gradient.
\begin{definition}[Gradient] The gradient, $\nabla f=\nabla^{(T,r,\rho)} f,$ of\/ $f\in{\mathcal A}$ is the
unique up to\/ $\lambda^{(T,r,\rho)}$-zero sets function\/ $g$ which satisfies (\ref{con.1}) for all\/ $x,y\in T$.
\label{Def:000}
\end{definition}\sm

\subsection{The regular Dirichlet form}
\label{Sub:form}
In this subsection we recall the construction of the so-called $\nu$-Brownian motion on an $\mathbb{R}$-tree  given in \cite{AthreyaEckhoffWinter2013}, and extend it to
arbitrary rooted metric measure trees.

As usual, we denote by ${\mathcal C}(T)$ the space of continuous functions $f\colon T\to\R$, and the subspace of
functions vanishing at infinity by
\begin{equation}
   {\mathcal C}_{\infty}(T) :=
   \bigl\{f \in {\mathcal C}(T):\,  \forall \varepsilon>0\; \exists\, K \text{ compact }\,\forall x \in
   T\setminus K :  |f(x)| \leq \varepsilon\bigr\}.
\end{equation}

Consider the bilinear form $({\mathcal E},{\mathcal D}({\mathcal E}))$
where
\begin{equation}
\label{p:005}
   {\mathcal E}(f,g)
    :=
   \tfrac{1}{2}\int\mathrm{d}\lambda\,\nabla f\nabla g,
\end{equation}
and
\begin{equation}
\label{p:004}
   {\mathcal D}({\mathcal E}):=
   	\big\{f\in L^2(\nu)\cap {\mathcal A}\cap{\mathcal C}_\infty(T) : \nabla f\in L^2(\lambda)\big\}.
\end{equation}

For technical purposes we also introduce for all closed subsets $A\subseteq T$ the domain
\begin{equation}
\label{e:222}
   {\mathcal D}_A({\mathcal E}):=\big\{f\in {\mathcal D}({\mathcal E}):\, f|_A\equiv 0\big\}.
\end{equation}

We first note that the bilinear form $({\mathcal E},{\mathcal D}_A({\mathcal E}))$ is closable for all closed
sets $A\subseteq T$.
Indeed, let $(f_n)_{n\in\mathbb{N}}$ be an ${\mathcal E}$-Cauchy sequence in $\D_A(\E)\subseteq L^2(\nu)$ with
$\|f_n\|_{L^2(\nu)}\to 0$. Then, by passing to a subsequence if necessary, we may assume $\nabla f_n\to 0$,
$\lambda^{(T,r,\rho)}$-almost surely and ${\mathcal E}(f_n,f_n)$ is uniformly bounded in $n\in\N$ (see for
example, \cite[(2.15),(2.16)]{AthreyaEckhoffWinter2013}).

Let $\({\mathcal E},\bar{{\mathcal D}}_A({\mathcal E})\)$ be the closure of $\(\E,\D_A(\E)\)$, i.e., $\bar{\D}_A(\E)$
is the closure of $\D_A(\E)$ with respect to $\E_1=\E+\langle\cdot,\cdot\rangle_\nu$.


\begin{remark}[Closing the form might not be  necessary]
The procedure of closing the form is unnecessary if the global lower mass-bound property holds on $T\setminus A$, i.e.,
for all $\delta>0$,
\begin{equation}
\label{e:global}
   \inf_{x\in T\setminus A}\nu\big(B(x,\delta)\big)>0.
\end{equation}
In this case, $\bar{{\mathcal D}}_A({\mathcal E})={\mathcal D}_A({\mathcal E})$.
\label{Rem:002}
\hfill$\qed$
\end{remark}\sm

The following lemma is an immediate consequence of Proposition~2.4, Lemma~2.8, Lemma~3.4, and Proposition~4.1 in
\cite{AthreyaEckhoffWinter2013}.

\begin{lemma}[Regular Dirichlet form] Let $(T,r,\nu)$ be a metric boundedly finite measure tree, and $A\subseteq T$ a closed subset.
Then the following hold:
\begin{itemize}
\item[(i)]
The bilinear form $({\mathcal E},\bar{\mathcal D}_A({\mathcal E}))$ is a regular Dirichlet form.
\item[(ii)] Dirac measures are of finite energy integral, there exists a constant $C_x>0$ such that for all $f\in{\mathcal D}({\mathcal E})\cap {\mathcal C}_0(T)$,
\begin{equation}\label{fei}
   f(x)^2\le C_x\,{\mathcal E}_1(f,f)
\ee
(See (2.2.1) in \cite{FukushimaOshimaTakeda1994}).
\item[(iii)] If\/ $A$ is non-empty the Dirichlet form is transient.
\end{itemize}
\label{Lem:001}
\end{lemma}\sm

It follows immediately from \cite[Theorem~7.2.1]{FukushimaOshimaTakeda1994} that there is a unique (up to $\nu$-equivalence) $\nu$-symmetric strong Markov process
\begin{equation}
\label{e:Xx}
   X=((X_t)_{t\ge 0},({\mathbb P}^x)_{x\in T})
\end{equation}
on $(T,r)$ associated with the regular Dirichlet form $({\mathcal E},
	\bar{{\mathcal D}}({\mathcal E}))$.

\begin{definition}[Speed-$\nu$ motion on $(T,r)$] Let\/ $(T,r,\nu)$ be a metric boundedly finite measure tree.
In the following we refer to the unique $\nu$-symmetric strong Markov process associated with\/ $({\mathcal E},\bar{{\mathcal D}}({\mathcal E}))$ as
the speed-$\nu$ motion on $(T,r)$.
\begin{itemize}
\item If\/ $(T,r)$ is discrete, then the speed-$\nu$ motion on $(T,r)$ is referred to as speed-$\nu$ random walk on $(T,r)$.
\item If\/ $(T,r)$ is an $\R$-tree, then the speed-$\nu$ motion on $(T,r)$ agrees with the $\nu$-Brownian motion on $(T,r)$ constructed in \cite{AthreyaEckhoffWinter2013}.
\end{itemize}
\label{Def:001}
\end{definition}\sm

\begin{remark}[Variable speed motion does not depend on root] Notice that although the definition of the length
measure and the gradient depend on the root, the Dirichlet form does not. Therefore  the variable speed motion
is independent on the choice of the root.
\label{Rem:006}
\hfill$\qed$
\end{remark}\sm

\begin{remark}[Connectedness and continuous paths]
	Notice that the Dirichlet form satisfies the local property if and only if the underlying space is
	connected. Thus the variable speed motion on $(T,r)$ has continuous paths if and only if\/ $(T,r)$ is an
	$\R$-tree.
\label{Rem:004}
\hfill$\qed$
\end{remark}\sm

Recall the explosion time $\zeta$ from (\ref{e:zeta}). Notice that the Dirichlet form need not be conservative, which means that the speed-$\nu$  motion
might exist only for a (random) finite {\em life time}.
This happens when it
explodes in finite time, i.e., $\zeta<\infty$.

\begin{remark}[Finite versus infinite life time] Let $(T,r,\nu)$ be a rooted boundedly finite measure tree, and $X$ the speed-$\nu$ motion on $(T,r)$.
Whether or not $\zeta=\infty$, almost surely, depends on the tree topology and the measure $\nu$.
\begin{enumerate}
\item The speed-$\nu$  motion on $(T,r)$ cannot explode if it is recurrent. Recurrence depends on $(T,r,\nu)$ only through $(T,r)$.
See \cite[Theorem~4]{AthreyaEckhoffWinter2013} for recurrence criteria.
\item An example of a transient variable speed motion with finite life time will be discussed in Example~\ref{ex:entrance}.
\hfill$\qed$
\end{enumerate}
\label{Rem:103}
\end{remark}\sm

\begin{lemma}[Variable speed motion on discrete trees is a Markov chain] Let $(T,r,\nu)$ be a metric boundedly
finite measure tree such that $(T,r)$ is discrete. Then the speed-$\nu$ random walk on $(T,r)$ is a continuous
time nearest neighbor Markov chain with jumps from $v$ to $v'\sim v$ at rate $\gamma_{vv'}:=\tfrac{1}{2\cdot\nu(\{v\})\cdot r(v,v')}$.
\label{L:009}
\end{lemma}\sm

\begin{proof}
Recall from Definition~\ref{Def:002} that $(T,r)$ is a Heine-Borel space. Thus each ball around $\rho$ contains
only a finite number of branch points, and in consequence the nearest neighbor random walk with the jump rates
$(\gamma_{vv'})_{v\sim  v'}$ is a well-defined strong Markov process. Its generator $\Omega$ acts on the space
${\mathcal C}_c(T)$ of continuous functions which depend only on finitely many $v\in T$ as follows:
\begin{equation}
\label{e:RWgen}
   \Omega f(v)
 :=
   \tfrac{1}{2\cdot\nu(\{v\})}\sum_{v'\sim v}\tfrac1{r(v,v')}\big(f(v')-f(v)\big).
\end{equation}

Notice that for all $f,g\in{\mathcal C}_c(T)$,
\begin{equation}
\label{e:heuristics}
\begin{aligned}
      {\mathcal E}(f,g)
   &=
      \tfrac{1}{2}\int\mathrm{d}\lambda\nabla f\nabla g
      \\
     &=
      \tfrac{1}{2}\sum_{v\in T}\tfrac{1}{2}\sum_{v'\sim v}\tfrac{1}{r(v,v')}\big(f(v')-f(v)\big)\big(g(v')-g(v)\big)
   \\
  &=
     -\sum_{v\in T}\nu(\{v\})\tfrac{1}{2\nu(\{v\})}\sum_{v'\sim v}\tfrac{1}{r(v,v')}\big(f(v')-f(v)\big)g(v)
   \\
 &=
    -\big(\Omega f,g\big)_\nu.
\end{aligned}
\end{equation}

The statement therefore follows from Example~1.2.5 together with Exercise~4.4.1 in \cite{FukushimaOshimaTakeda1994}.
\end{proof}\sm

\subsection{The occupation time formula}
\label{Sub:occupation}
We conclude this section by recalling here the occupation time formula as known from speed-$\nu$ motions on $\R$ or the $\nu$-Brownian motion on compact metric trees (see, for example, \cite[Proposition~1.9]{AthreyaEckhoffWinter2013}).

As usual, we denote for each $x\in T$ by
\begin{equation}
\label{e:hitting}
   \tau_x=\tau_x(X):=\inf\set{t\ge 0}{X_t=x}
\end{equation}
the {\em first hitting time} of $x$.
A standard calculation shows the following:
\begin{proposition}[Occupation time formula] Let $X$ be a speed-$\nu$ motion on $(T,r)$.
If\/ $X$ is recurrent, then for all $x,z\in T$,
\begin{equation}
\label{eq:occupation}
		\mathbb{E}^x\Bigl[\int_0^{\tau_z}f(X_t)\,\mathrm{d}t\Bigr] = 2\int_{T}f(y)\cdot r\big(z, c(x,z,y)\big)\,\nu(\mathrm{d}y),
	\end{equation}
for all bounded, measurable $f\colon T\to\R$. Moreover, the process $X_{\boldsymbol{\cdot}\wedge\tau_z}$ is transient for all $z\in T$.
\label{P:007}
\end{proposition}\sm

\begin{proof} Let $(T,r,\nu)$ be a metric boundedly finite measure tree, and $z\in T$ fixed. By Lemma~\ref{Lem:001}(iii), the Dirichlet form
$({\mathcal E},{\mathcal D}_{\{z\}}({\mathcal E}))$ is transient.  Therefore by Theorem~4.4.1(ii) in \cite{FukushimaOshimaTakeda1994} ,
$R_{\{z\}}f(x):=\mathbb{E}^{x}[ \int_{0}^{\tau_{z}}\mathrm{d}s\, f(X_{s})]$ is the resolvent of the speed-$\nu$ motion killed on hitting $z$, i.e.,
\begin{equation}
\label{e:031}
  {\mathcal E}(R_{\{z\}}f,h) = \int  \mathrm{d}\nu\, h\cdot f,
\end{equation}
for all $h \in \bar{\mathcal D}_{\{z\}}(\mathcal E)$ and $f\in{\mathcal D}({\mathcal E})$ with $(R_{\{z\}}f,f)_\nu<\infty$.
The resolvent of a Markov process has the form
\begin{equation}
\label{e:resolcap}
   R_{\{z\}}f(x)=\int_T\nu(\mathrm{d}y)\,\tfrac{h^\ast_{\{z\},y}(x)}{\mathrm{cap}_{\{z\}}(y)}f(y),
\end{equation}
where $\mathrm{cap}_{\{z\}}(y):=\inf\{{\mathcal E}(f,f):\,f\in\bar{\mathcal D}({\mathcal
E}),\,f(z)=0,\,f(y)=1\}$ and $h^\ast_{\{z\},y}$ is the unique minimizer for $\mathrm{cap}_{\{z\}}(y)$.
This can be shown by essentially rewriting the argument laid out in  \cite[Section~3]{AthreyaEckhoffWinter2013}. Moreover, for our particular Dirichlet form we find that  $h^\ast_{\{z\},y}(x):=\tfrac{r(c(x,y,z),z)}{r(y,z)}$ and $\mathrm{cap}_{\{z\}}(y)=\tfrac{1}{2r(y,z)}$,
and thus that
\begin{equation}
\label{e:021}
   \mathbb{E}^{x}\Bigl[ \int_{0}^{\tau_{z}}\mathrm{d}s\, f(X_{s})\Bigr]
 =
   2\int\nu(\mathrm{d}y)\,r(z,c(x,y,z))f(y). \qedhere
\end{equation}
\end{proof}\sm

\section{Preliminaries on the Gromov-vague topology}
\label{S:topology}
Recall the notion of a rooted metric (boundedly finite) measure space $(T,r,\rho,\nu)$ from Definition~\ref{Def:002}. Once more, we call two rooted metric  measure trees $(T,r,\rho,\nu)$ and $(T',r',\rho',\nu')$ equivalent iff there is an isometry $\varphi$ between $\supp(\nu)\cup\{\rho\}$ and $\supp(\nu')\cup\{\rho'\}$ such that $\varphi(\rho)=\rho'$ and $\nu\circ\varphi^{-1}=\nu'$, and denote by
\begin{equation}
\label{e:mathbbT}
   \mathbb{T}:=\mbox{ the space of equivalence classes of rooted metric measure trees}.
\end{equation}

In this section we want to equip $\mathbb{T}$  with the so-called {\em Gromov-Hausdorff-vague topology} on which the convergence of the underlying spaces in our invariance principle is based.
We refer the reader to \cite{ALW2} for many detailed discussions.
We recall the definition of the pointed Gromov-weak topology on finite metric measure spaces in Subsection~\ref{Sub:Gweak} and then extend it
to a Gromov-vague topology on $\T$ in Subsection~\ref{Sub:Gvague}.
Finally we compare the notions of Gromov-weak and Gromov-vague convergence in Subsection~\ref{Sub:Gwversusv}.

\subsection{Gromov-weak and Gromov-Hausdorff-weak topology}
\label{Sub:Gweak}
In this subsection we restrict to compact metric spaces and recall the Gromov-weak topology.
This topology originates from the work of
Gromov \cite{MR2000d:53065} who considers topologies allowing to compare metric spaces who might not be
subspaces of a common metric space. The {\em Gromov-weak topology} on complete and separable metric measure
spaces was introduced in \cite{GrevenPfaffelhuberWinter2009}. In the same paper the Gromov-weak topology was
metrized by the so-called {\em Gromov-Prohorov-metric} which is equivalent to Gromov's box metric introduced in
\cite{MR2000d:53065}, as was shown in \cite{Loehr}.
The topology is closely related to
the so-called {\em  measured Gromov-Hausdorff topology} which was first introduced by \cite{Fukaya1987}, and further discussed in \cite{KuwaeShioya2003,EvansWinter2006}.

\begin{remark}[Full-support assumption]
	Note that, by our definition, the measure $\nu$ of a metric boundedly finite measure tree $(T, r, \rho,
	\nu)$ is required to have full support. This is usually not assumed for metric measure spaces, but it is
	only a minor restriction, because, whenever $\rho \in \supp(\nu)$ we can choose representatives with
	full support.
 \label{Rem:050}
\hfill$\qed$
\end{remark}\sm

 Consider also the subspace
\begin{equation}
\label{e:mathbbRT}
   \mathbb{T}_{c}:=\big\{(T,r,\rho,\nu)\in\mathbb{T}:\,(T,r)\mbox{ is compact}\big\}.
\end{equation}

We shortly recall the basic definitions of the Gromov-weak and Gromov-Hausdorff-weak topologies on $\mathbb{T}_c$.
\begin{definition}[Gromov-weak and Gromov-Hausdorff-weak topology]
Let for each $n\in\mathbb{N}\cup\{\infty\}$, $\tree_n:=(T_n,r_n,\rho_n,\nu_n)$ be
in $\mathbb{T}_c$. We say that $(\tree_n)_{n\in\mathbb{N}}$ converges to $\tree_\infty$ in
\begin{itemize}
\item[(i)] \emph{pointed Gromov-weak topology} if and only if there exists a complete, separable rooted metric
space $(E,d_E,\rho_E)$ and for each $n\in\mathbb{N}\cup\{\infty\}$ isometries $\varphi_n:T_n\to E$ with
$\varphi_n(\rho_n)=\rho_E$, and such that
\begin{equation}
\label{e:Gweak}
   (\varphi_n)_*\nu_n\Tno (\varphi_\infty)_*\nu_\infty.
\end{equation}
\item[(ii)]  \emph{pointed Gromov-Hausdorff-weak topology} if and only if there exists a compact metric space $(E,d_E,\rho_E)$
and for each $n\in\mathbb{N}\cup\{\infty\}$ isometries $\varphi_n\colon T_n\to E$ with $\varphi_n(\rho_n)=\rho_E$,
such that \eqref{e:Gweak} holds and
\begin{equation}
\label{e:GHweak}
	\supp\((\varphi_n)_\ast\nu_n\) \,\stackrel{\mbox{\tiny Hausdorff}}{\tno}\, \supp\((\varphi_\infty)_\ast\nu_\infty\).
\end{equation}
\end{itemize}
\label{Def:004}
\end{definition}\sm

\begin{remark}[Supports do not converge under Gromov-weak convergence]
Consider, for example, $T_n:\equiv\{\rho,\rho'\}$ and $r_n(\rho,\rho')\equiv 1$, and put $\nu_n:=\tfrac{n-1}{n}\delta_{\rho}+\tfrac{1}{n}\delta_{\rho'}$  for all $n\in\mathbb{N}$.
Clearly, $\((T_n,r_n,\rho,\nu_n)\)_{n\in\mathbb{N}}$ converges pointed Gromov-weakly to the unit mass pointed singleton $(\{\rho\},\rho,\delta_\rho)$.
The supports, however, do not converge. This shows that
Gromov-weak is in general weaker than Gromov-Hausdorff-weak convergence.
\label{Rem:007}
\hfill$\qed$
\end{remark}\sm

In order to close the gap between Gromov-weak and Gromov-Hausdorff-weak convergence,
we define for each $\delta>0$
the {\em lower mass-bound function} $m_\delta\colon \mathbb{T}\to \R_+$ as
\begin{equation}
\label{e:mdelta}
	m_\delta\big((T, r, \rho, \nu)\big) := \inf\bset{\nu\(\Bcl_r(x,\delta)\)}{x\in T}.
\end{equation}
It follows from our full-support assumption, $\supp(\nu)=T$, that $m_\delta(\tree)>0$ for all $\delta>0$ if
$\tree\in\mathbb{T}_c$.

\begin{definition}[Global lower mass-bound property] We say that a family $\Gamma\subseteq\mathbb{T}_c$ satisfies the
global lower mass-bound property if and only if the lower mass-bound functions are all bounded away from zero uniformly in $\Gamma$, i.e.,
for each $\delta>0$,
\begin{equation}
\label{e:glmb}
   m_\delta\big(\Gamma\big):=\inf_{\tree\in\Gamma}m_\delta(\tree)>0.
\end{equation}
\label{Def:005}
\end{definition}\sm

The following is Theorem~6.1 in \cite{ALW2}.
\begin{proposition}[Gromov-weak versus Gromov-Hausdorff-weak topology]  Let for each $n\in\mathbb{N}\cup\{\infty\}$,
$\tree_n:=(T_n,r_n,\rho_n,\nu_n)$ be
in $\mathbb{T}_c$ such that $(\tree_n)_{n\in\mathbb{N}}$ converges to $\tree_\infty$ pointed Gromov-weakly. Then the following are equivalent:
\begin{itemize}
\item[(i)]  The sequence $(\tree_n)_{n\in\mathbb{N}}$ converges to $\tree_\infty$ pointed Gromov-Hausdorff-weakly.
\item[(ii)] The sequence $(\tree_n)_{n\in\mathbb{N}}$ satisfies the global lower mass-bound property.
\end{itemize}
\label{P:002}
\end{proposition}\sm

\subsection{Gromov-vague and Gromov-Hausdorff-vague topology}
\label{Sub:Gvague}
\hfill\newline
Recently, in \cite{AbrahamDelmasHoscheit2013}, the Gromov-Hausdorff-weak topology on rooted compact length spaces
was extended to complete locally compact length spaces equipped with locally finite measures. In this subsection
we want, in similar spirit, extend the Gromov(\nbd Hausdorff)-weak topology on $\mathbb{T}_c$ to the
Gromov(\nbd Hausdorff)-vague topology on $\mathbb{T}$.

The restriction of $\smallx=(X, r, \rho,\nu)\in\mathbb{T}$ to the closed ball $\Bcl(\rho, R)$
of radius $R> 0$ around the root is denoted by
\begin{equation}
\label{e:restrict}
	\smallx\restricted{R}:=\(\Bcl(\rho, R),r,\rho,\nu\restricted{\Bcl_r(\rho, R)}\).
\end{equation}

\begin{definition}[Gromov-vague topology] Let for each $n\in\mathbb{N}\cup\{\infty\}$,
$\tree_n:=(T_n,r_n,\rho_n,\nu_n)$ be
in $\mathbb{T}$. We say that $(\tree_n)_{n\in\mathbb{N}}$ converges to $\tree_\infty$ in
\begin{itemize}
\item[(i)] {\em pointed Gromov-vague topology} if and only if there exists a complete, separable rooted metric space $(E,d_E,\rho_E)$
and for each $n\in\mathbb{N}\cup\{\infty\}$ isometries $\varphi_n:T_n\to E$ with $\varphi_n(\rho_n)=\rho_E$, and
such that
\begin{equation}
\label{e:Gweak2}
   \bigl((\varphi_n)_\ast\nu_n\bigr)\restricted{R}\Tno\big((\varphi_\infty)_\ast\nu_\infty\big)\restricted{R}
\end{equation}
for all but countably many $R>0$.
\item[(ii)]  {\em pointed Gromov-Hausdorff-vague topology} if and only if there exists a rooted Heine-Borel
space $(E,d_E,\rho_E)$ and for each $n\in\mathbb{N}\cup\{\infty\}$ isometries $\varphi_n:T_n\to E$ with
$\varphi_n(\rho_n)=\rho_E$, and such that \eqref{e:Gweak2} and
\begin{equation}
\label{e:Gweak3}
   \varphi_n(T_n)\cap \Bcl_{d_E}(\rho_E, R) \stackrel{\mbox{\tiny Hausdorff}}{\tno} \varphi(T)\cap \Bcl_{d_E}(\rho_E, R)
\end{equation}
hold for all but countably many $R>0$.
\end{itemize}
\label{Def:Gv}
\end{definition}\sm

Once more we want to close the gap between Gromov-vague and Gromov-Hausdorff-vague convergence. Define therefore for all
$\delta>0$ and $R>0$,
the {\em local lower mass-bound function} $m^R_\delta\colon \mathbb{T}\to \R_+\cup\{\infty\}$ as
\begin{equation}
\label{e:mdeltaR}
	m^R_\delta\big((T, r, \rho, \nu)\big) := \inf\bset{\nu\(\Bcl_r(x,\delta)\)}{x\in B(\rho,R)}.
\end{equation}
Notice that $m^R_\delta(\tree)>0$ for all $\tree\in\mathbb{T}$, and $\delta,R>0$.

\begin{definition}[Local lower mass-bound property] We say that a family $\Gamma\subseteq\mathbb{T}$ satisfies the
local lower mass-bound property if and only if the lower mass-bound functions are all bounded away from zero uniformly in $\Gamma$, i.e.,
for each $\delta>0$ and $R>0$,
\begin{equation}
\label{e:033}
   m^R_\delta\big(\Gamma\big):=\inf_{\tree\in\Gamma}m^R_\delta(\tree)>0.
\end{equation}
\label{Def:006}
\end{definition}\sm

The following is Corollary~5.2 in \cite{ALW2}.
\begin{proposition}[Gromov-vague versus Gromov-Hausdorff-vague]
 Let for each $n\in\mathbb{N}\cup\{\infty\}$,
$\tree_n:=(T_n,r_n,\rho_n,\nu_n)$ be
in $\mathbb{T}$ such that $(\tree_n)_{n\in\mathbb{N}}$ converges to $\tree_\infty$ pointed Gromov-vaguely. Then the following are equivalent:
\begin{itemize}
\item[(i)]  The sequence $(\tree_n)_{n\in\mathbb{N}}$ converges to $\tree_\infty$ pointed Gromov-Hausdorff-vaguely.
\item[(ii)] The sequence $(\tree_n)_{n\in\mathbb{N}}$ satisfies the local lower mass-bound property.
\end{itemize}
\label{P:013}
\end{proposition}\sm

\subsection{Gromov-weak versus Gromov-vague convergence}
\label{Sub:Gwversusv}
Note that the concept of Gromov-vague convergence on $\mathbb{T}$ is not strictly an extension
of the concept of Gromov-weak convergence on $\mathbb{T}_c$
because in the limit parts might ``vanish at infinity'',
and hence a non-converging sequence of compact spaces with respect to the  Gromov-weak or the Gromov-Hausdorff-weak topology may
converge in the ``locally compact version'' of the corresponding topology.

\begin{remark}[Gromov-vague versus Gromov-weak]
	Consider the subspaces $\mathbb{T}_{\mbox{\tiny finite}}$ and $\mathbb{T}_{\mbox{\tiny probability}}$ of $\mathbb{T}$ consisting of spaces
	$\tree=(T,r,\rho,\nu)\in\mathbb{T}$ where $\nu$ is a finite or a probability measure, respectively.
	Then on $\mathbb{T}_{\mbox{\tiny probability}}$ the induced Gromov-vague topology coincides with the Gromov-weak topology.
	However, on $\mathbb{T}_{\mbox{\tiny finite}}$ and even on $\mathbb{T}_c$ this is
	not the case as the total mass might not be preserved under  Gromov-vague convergence.
	In fact, for $\tree, \tree_n=(T_n, r_n,\rho_n, \nu_n) \in\mathbb{T}_{\mbox{\tiny finite}}$ the following are equivalent:
\begin{enumerate}
\item
	$\tree_n \to \tree$ Gromov-weakly.
\item $\tree_n \to \smallx$ Gromov-vaguely and\/ $\nu_n(T_n) \to \nu(T)$.
\end{enumerate}
Moreover, $\treen \to \tree \in \mathbb{T}_c$ Gromov-Hausdorff-weakly if and only if\/ $\treen\to\tree$
Gromov-Hausdorff-vaguely and the diameters of\/ $(T_n, r_n)$ are bounded uniformly in $n$ (except for finitely
many $n$).
\label{Rem:008}
\hfill$\qed$
\end{remark}\sm

\section{Tightness}
\label{S:tight}
Recall the speed-$\nu$ motion on $(T,r)$, $X^{(T,r,\nu)}$, from Definition~\ref{Def:001}.
In this section we prove that the sequence $\{X^{(T_n,r_n,\nu_n)};\,n\in\mathbb{N}\}$ is tight provided that
Assumptions~(A0), (A1) and (A2) from Theorem~\ref{T:001} are satisfied. The main result is the following:

\begin{proposition}[Tightness]
Let\/ $\tree:=(T,r,\rho,\nu)$ and\/ $\tree_n:=(T_n,r_n,\rho_n,\nu_n)$, $n\in\N$, be rooted metric boundedly
finite measure trees. Assume that for all\/ $n\in\mathbb{N}$, $\tree_n$ is discrete, and that the following
conditions hold:
\begin{itemize}
	\item[(A0)] For all\/ $R>0$,
	\begin{equation}
	\label{e:019}
		\limsup_{n\to\infty}\sup\bset{r_n(x,z)}{x\in B_n(\rho_n,R),\, z\in T_n,\, x\sim z}<\infty.
	\end{equation}
	\item[(A1)] The sequence\/ $(\tree_n)_{n\in\mathbb{N}}$ converges to\/ $\tree$ in the pointed Gromov-vague
		topology as\/ $n\to\infty$.
	\item[(A2)] The local lower mass-bound property holds uniformly in\/ $n\in\mathbb{N}$.
\end{itemize}
Then there is a Heine Borel space $(E,d)$, such that\/ $T$ and all\/ $T_n$, $n\in\N$, are embedded in $(E,d)$ and
the sequence $X^n$, $n\in\N$, of speed-$\nu_n$ random walks on $(T_n,r_n)$ is tight in the one-point
compactification of\/ $E$.
\label{prop:tight}
\end{proposition}\sm

For the proof we rely on the following version of the Aldous tightness criterion
(see, \cite[Theorem~16.11+16.10]{Kallenberg2001}).
\begin{proposition}[Aldous tightness criterion]
	Let $X^n=(X^n_t)_{t\ge0}$, $n\in\N$, be a sequence of \cadlag\ processes on a complete, separable metric
	space $(E,d)$. Assume that the one-dimensional marginal distributions are tight, and for any bounded
	sequence of $X^n$-stopping times $\tau_n$ and any $\delta_n>0$ with $\delta_n \to 0$ we have
	\begin{equation}
    \label{eq:Kallenberg}
		d\bigl(X^n_{\tau_n}, X^n_{\tau_n+\delta_n}\bigr) \ton 0\; \text{ in probability.}
	\end{equation}
	Then the sequence $(X^n)_{n\in\N}$ is tight.
\label{P:Aldous}
\end{proposition}\sm

To verify Proposition~\ref{P:Aldous}, we have to show that it is unlikely that the walk has moved more than a
certain distance in a sufficiently small amount of time, uniformly in $n$ and the starting point.
\begin{cor}\label{cor:Aldous}
	Let\/ $(E,d)$ be a locally compact, separable metric space. For each\/ $n\in\N$, let\/ $T_n\subseteq
	E$ and\/ $(X^n,(\mathbb{P}^{x})_{x\in T_n})$ a strong Markov process on\/ $T_n$.
	Assume that for every\/ $\varepsilon>0$
	\begin{equation}
\label{eq:toshow}
		\lim_{t \to 0} \lim_{n\to\infty} \sup_{x\in T_n} \mathbb{P}^{x}\big\{d(x,X^n_t)>\varepsilon\big\} = 0.
	\end{equation}
	Then for every sequence of initial distributions\/ $\mu_n\in \CM_1(T_n)$ the sequence\/
	$(X^n)_{n\in\N}$ is tight as processes on the one-point compactification of\/ $E$.
\end{cor}\sm

\begin{proof}
	Let $\Eh= E \cup\{\infty\}$ be the one-point compactification of $E$. $\Eh$ is metrizable, and we can
	choose a metric $\dh$ with $\dh\le d$ on $E\times E$. A possible choice is
	\begin{equation}
\label{e:dd}
   \dh(x,y)
 :=
   \inf_{n\in\N} \inf_{z_1,\ldots,z_n\in E}
		\sum_{k=0}^n e^{-\inf\set{j}{z_k\in U_j \text{ or } z_{k+1}\in U_j}} \(1\land d(z_k, z_{k+1})\),
\end{equation}
	where $z_0:=x$, $z_{n+1}:=y$, and $U_1\subseteq U_2 \subseteq \cdots$ are (fixed) open, relatively
	compact subsets of $E$ with $E=\bigcup_{n\in\N} U_n$.
	By the strong Markov property, \eqref{eq:toshow} implies \eqref{eq:Kallenberg} for $d$ and hence also
	for $\dh$. By \propref{Aldous}, $(X^n)_{n\in\N}$ is tight on $\Eh$.
\end{proof}\sm

From here we proceed in several  steps. We first give an estimate for the probability to reach a particular point in a small amount of time.
We are then seeking an estimate for the probability that the walk has moved more than a given distance away from the starting point.
For that we will need a bound on the number of possible directions the random walk might have taken until reaching that distance.

Recall from (\ref{e:hitting}) the first hitting time $\tau_x$ of a point $x\in T$.

\begin{lemma}[Hitting time bound]
	Let\/ $(T, r,\rho,\nu)$ be a discrete rooted metric boundedly finite measure tree, $x\in T$, $X$ the
	speed-$\nu$ random walk on\/ $(T,r)$ started at\/ $x$. Fix\/ $v\in T$ and\/ $\delta\in (0,r(x,v))$.
	Denote by\/ $S:=B(x,\delta)$ the subtree\/ $\delta$-close to\/ $x$ and let\/ $R:= r(S,v)$.
	Then, for all\/ $t\ge 0$,
	\begin{equation}
\label{eq:hit}
		\mathbb{P}^x\big\{\tau_v\le t\big\}\le 2\Big(1 - \tfrac{R}{R+2\delta} e^{-\frac{t}{R\nu(S)}} \Big).
	\end{equation}
\label{Lem:hit}
\end{lemma}\sm

\begin{proof}
	Assume w.l.o.g.\ that $X$ is recurrent and let $w$ be the unique point in $S$ with $r(v, w)=R$.
	Obviously, if $X$ starts in $x$ then it must pass $w$ before hitting $v$. Neglect the time until $w$ and
	assume $X$ starts in $w$ instead of $x$.
	For $u\in S$, let $t_u$ be the (random) amount of time spent in $u$ before hitting $v$, $r_u:=r(w,
	u)$, and $m_u := \nu\(\{u\}\)$. Using that a geometric sum of independent, exponentially distributed
	random variables is again exponentially distributed, it is easy to see that the law of $t_u$ is
	\begin{equation} \label{eq:tudist}
		\mathcal{L}^w(t_u) = \tfrac{r_u}{R+r_u}\delta_0 + \tfrac{R}{R+r_u}\Expo\(\tfrac1{2(R+r_u)m_u}\),
	\end{equation}
	where $\Expo(\lambda)$ denotes an exponential distribution with expectation $\frac1\lambda$, and
	$\delta_0$ the Dirac measure in $0$.

As $\tau_v\ge\sum_{u\in S}t_u$, we find that for every $a>0$,
\begin{equation}
\label{e:020}
   \tau_v \ge \sum_{u\in S} a m_u \indicator{\{t_u\ge a m_u\}} =
    a \,\nu\(\{u\in S:\,t_u\ge am_u\}\). 
\end{equation}
Now we pick $a:=\frac{2t}{\nu(S)}$ and obtain
\begin{equation}
\label{e:060}
\begin{aligned}
   \mathbb{P}^x\bigl\{\tau_v\le t\bigr\}
  &\le
    \mathbb{P}^w\bigl\{\nu\{u\in S:\,t_u\ge am_u\}\le\tfrac{1}{2}\nu(S)\bigr\}
    \\
  &=
    \mathbb{P}^w\bigl\{\nu\{u\in S:\,t_u < am_u\}\ge\tfrac{1}{2}\nu(S)\bigr\}
    \\
  &\le
    \tfrac{2}{\nu(S)}\mathbb{E}^w\bigl[\nu\{u\in S:\,t_u<am_u\}\bigr]
    \\
  &=
     \tfrac{2}{\nu(S)}\sum_{u\in S}m_u\mathbb{P}^w\big\{t_u<2t\tfrac{m_u}{\nu(S)}\big\},
\end{aligned}
\end{equation}
which together with (\ref{eq:tudist}) and the fact that $r_u \le 2\delta$ gives the claim.
\end{proof}\sm

To get bounds on the probability to move sufficiently far from bounds on the probability to hit a pre-specified
point, we need a bound on the number of directions the random walk can take in order to get far away. With
$\eps$\nbd degree of a node $x$ we mean the number of edges that intersect the $\eps$-sphere around $x$ and are
connected to points at least $2\eps$ away from $x$.

\begin{definition}[$\eps$-degree]\label{def:deg}
	Let\/ $(T,r)$ be a discrete metric tree.
	For $\eps>0$, $x\in T$, let $B:=B(x,\eps)$ be the $\eps$-ball around $x$.
	The $\eps$-degree of $x$ is
	\begin{equation}
\begin{aligned}
\label{e:degx}
		&\deg_\eps(x) := \deg_\eps^T(x) \\
		&:=\#\big\{v\in T\setminus B:\,\exists u\in B, w\in T\setminus B(x,2\eps):
			u\sim v,\, v\in[u,w]\big\}.
	\end{aligned}
\end{equation}
	We also define the maximal degree as
	\begin{equation}
\label{e:degT}
		\deg_\eps(T) := \sup_{x\in T} \deg^T_\eps(x).
	\end{equation}
\end{definition}\sm

\begin{lemma}[Topological bound]\label{Lem:deg}
    Let\/ $\smallt_n:=(T_n,r_n)$, $n\in\mathbb{N}$, be discrete metric trees, and\/ $\smallt:=(T,r)$ a compact
    metric tree.
    If\/ $(\smallt_n)_{n\in\mathbb{N}}$ converges to\/ $\smallt$  in Gromov-Hausdorff topology, then for every\/
    $\eps>0$,
	\begin{equation}
		\limsup_{n\to \infty} \deg_\eps(T_n) < \infty.
	\end{equation}
\end{lemma}\sm

\begin{proof} Fix $\eps>0$.
As $\smallt_n\to\smallt$ in Gromov-Hausdorff topology, there exists a finite $\eps$-net $S$ in
$T$, and $\eps$-nets $S_n$ in $T_n$, such that for all sufficiently large $n\in\N$, $S_n$ has the same cardinality as $S$
(see, for example, \cite[Proposition~7.4.12]{BurBurIva01}).
Obviously, this common cardinality is an upper bound for $\{\deg_\eps(T_n);\,n\in\mathbb{N}\}$.
\end{proof}\sm

With the notion of an $\varepsilon$-degree of a tree, we can immediately conclude the following.
\begin{lemma}[Speed bound]\label{Lem:speed}
	Let\/ $(T, r,\rho, \nu)$ be a discrete metric boundedly finite measure tree, $x\in T$, and\/ $X$ the
	speed-$\nu$ random walk on\/ $(T,r)$.
    Then for every\/ $\eps>0$, $\delta\in(0,\eps)$ and\/ $t<(\eps-\delta)m$, where\/ $m:=\nu\(B(x,\delta)\)$,
	\begin{equation}
		\mathbb{P}^x\big\{\sup_{s\in[0,t]}r(X_s, x)>2\eps\big\}
  \le
		2\deg_\eps(x)\Bigl(1-\tfrac{\eps-\delta}{\eps+\delta}\exp\(-\tfrac{t}{\eps m}\)\Bigr).
	\end{equation}
\end{lemma}\sm

\begin{proof}
	Let $v_1,\ldots,v_N$ be the points outside $B(x,\eps)$ that are neighbours of a point inside $B(x,\eps)$ and
	on the way from $x$ to a point outside $B(x,2\eps)$. Then $N\le \deg_\eps(x)$.
	Under $\mathbb{P}^x$, if $r(X_s, x) > 2\eps$ for some $s\le t$, $X$ must have hit at least one point in
	$\{v_1,\ldots,v_N\}$ before time $s$. Hence the claim follows from Lemma~\ref{Lem:hit}.
\end{proof}\sm

\begin{proof}[Proof of Proposition~\ref{prop:tight}]\label{proofoftight}
	According to Proposition~\ref{P:013}, $\smallt_n \to \smallt$ in Gromov-Hausdorff-vague topology. Hence,
	we may assume that there is a rooted Heine-Borel space $(E, d, \rho_E)$, such that
	$T_n,T\subseteq E$, $\rho_E=\rho=\rho_n$ for all $n\in\N$, and, for all but countably many $R>0$, we
	have both
	\begin{equation}
		T_n\cap \Bcl_d(\rho, R) \to T\cap \Bcl_d(\rho,R)
	\end{equation}
	as subsets of $E$ in Hausdorff topology, and
	\begin{equation}
		\nu_n\restricted{R} \Rightarrow \nu\restricted{R}.
	\end{equation}
	Let $\Eh=E\cup\{\infty\}$ be the one-point compactification of $E$, metrized by a metric $\dh$ with
	$\dh\le d$ on $E^2$ (see, for example, (\ref{e:dd})). For each $x\in\Eh$ and $N\in\mathbb{N}$, write
    $B_{\dh}(x,\frac{1}{N}):=\{y\in\Eh:\,\dh(x,y)<\frac{1}{N}\}$ and put
 \begin{equation}
 \label{e:KN}
     K_N
  :=
     \Eh\setminus B_{\dh}(\infty,\tfrac1N)\subseteq E.
 \end{equation}
 Notice that $K_N$ is compact by definition.

	To show tightness, we show that condition \eqref{eq:toshow} of Corollary~\ref{cor:Aldous} is satisfied
	for the metric $\dh$, i.e., for given $\eps, \epsh>0$, we can construct $t_0>0$ such that
	\begin{equation}\label{eq:claim}
		\sup_{x\in T_n} \mathbb{P}^x\big\{\dh(x, X^{n}_t) > \eps\big\} \le \epsh,
	\end{equation}
for all $t\in[0,t_0]$ and all $n\in\N$.

	Fix $\eps>0$, and choose $N>\frac{4}{\eps}$. Then the diameter of $\Eh\setminus K_N$ with respect to $\dh$ is at most
	$\frac12\eps$. Let
\begin{equation}
\label{e:eN}
    e_N
   :=
     \sup_{n\in\N}\sup_{x\in T_n\cap K_N}\sup_{y\sim x} d(x,y)
\end{equation}
     be the supremum
	of edge-lengths emanating from points in $T_n\cap K_N$, and note that $e_N<\infty$ by assumption.
	Now choose $M>N$ such that $K_M$ contains the $e_N$-neighbourhood of $K_N$, i.e.,
    $\{x'\in E:\,d(K_N,x')<e_N\}\subseteq K_M$. Then all points of $K_M$ which are connected to a point in $E\setminus K_M$ (within some
	$T_n$) are actually in $K_M\setminus K_N$.

Consider the hitting time of $K_M$, $\tau_{K_M}:=\inf\set{s\ge 0}{X^n_s\in K_M}$, and recall that the $\dh$\nbd
diameter of $\Eh\setminus K_N$ is at most $\frac\eps2$.
Therefore, if $X^n$ starts in $x\in T_n$, then $\dh(x, X_t^n) > \eps$ implies $\tau_{K_M} < t$ and
$\dh(x,X_{\tau_{K_M}}^n) \le \frac\eps2$.
Using the strong Markov property at $\tau_{K_M}$, we
	obtain for all $n\in\N$, $x\in T_n$,
	\begin{equation}
\begin{aligned}
		\mathbb{P}^x\big\{\dh(x,X^{n}_t)>\eps\big\}
			&\le
    \sup_{y\in T_n\cap K_M} \sup_{s\in[0,t]} \mathbb{P}^y\big\{ \dh(y, X^{n}_s) > \tfrac12\eps\big\}.
\label{eq:42.42}
\end{aligned}
	\end{equation}

Applying Lemma~\ref{Lem:speed}, we conclude for all $\delta\in(0,\eps)$ and $t<\frac14(\eps-\delta)m_\delta$, where
	$m_\delta:=\inf_{n\in\mathbb{N}} \inf_{y\in T_n\cap K_M} \nu_n\(B(y, \frac\delta4)\)$,
	\begin{equation}
\begin{aligned}
\mathbb{P}^x\big\{\dh(x,X^{n}_t)>\eps\big\}
			&\le 2\deg_{\frac\eps4}(T_n\cap K_M)\big(1-\tfrac{\eps-\delta}{\eps+\delta}
				\exp\(-\tfrac{4t}{\eps m_\delta}\)\big).
\label{eq:42}
\end{aligned}
	\end{equation}

	As $D:=\sup_{n\in\N} \deg_{\frac\eps4}(T_n\cap K_M) < \infty$ by Lemma~\ref{Lem:deg}, and $m_\delta>0$
	by the local lower mass-bound property (A2), we can choose $\delta>0$ small enough such that
	$\tfrac{\eps-\delta}{\eps+\delta} > 1- \tfrac{\epsh}{4D}$, and subsequently
	$t_0<\frac14(\eps-\delta)m_\delta$
	such that $\exp({-\frac{4t_0}{\eps m_\delta}}) > 1- \tfrac{\epsh}{4D}$. Inserting this into
	\eqref{eq:42}, we obtain \eqref{eq:claim} and tightness follows from \corref{Aldous}.
\end{proof}\sm

\section{Identifying the limit}
\label{S:ident}
In this section we identify the limit process. For this purpose, we use a characterization from
\cite[Section~5]{MR93f:60010}, where the existence of a diffusion process on a particular non-trivial continuum
tree, the so-called Brownian CRT $(T,r,\nu)$ from Example~\ref{Exp:002}, was shown. Aldous defines this diffusion as a strong Markov process on $T$ with continuous path such that $\nu$ is the reversible equilibrium  and it satisfies
the following two properties:
\begin{itemize}
\item[(i)] For all $a,b,x\in T$ with $x\in[a,b]$, $\mathbb{P}^x\{\tau_a<\tau_b\}=\frac{r(x,b)}{r(a,b)}$.
\item[(ii)] The occupation time formula (\ref{e:abstract}) holds.
\end{itemize}
While (i) reflects the fact that this diffusion is on ``natural scale'', (ii) recovers $\nu$ as the ``speed'' measure.
At several places in the literature constructions of diffusions on the CRT and more general continuum random trees
rely on Aldous' characterisation (see, for example,
\cite{Kre95,Croydon2008,Croydon2010}).
Albeit the diffusions can be indeed characterised by (i) and (ii) uniquely, a formal proof for this fact has to the best of our knowledge
never been given anywhere. We want to close this gap, and even show that the requirement (i) is redundant.

The following result will be proven in Subsection~\ref{Sub:compact}.
\begin{proposition}[Characterization via occupation time formula]
Assume that $(T,r)$ is a compact metric tree, and that we are given two $T$-valued strong
Markov processes $X$ and $Y$ such that for all $x,y\in T$, and bounded measurable $f\colon T \to \R_+$,
\begin{equation}
\label{eq:occueqal}
		\mathbb{E}^x\big[\int_0^{\tau_y}\mathrm{d}t\,f(X_t)\big]
		=\mathbb{E}^x\big[\int_0^{\tau_y}\mathrm{d}t\,f(Y_t)\big].
	\end{equation}
Assume further that $X_{\boldsymbol{\cdot}\wedge\tau_y}$ is transient for all $y\in T$.
Then the laws of $X$ and $Y$ agree.
\label{P:001}
\end{proposition}\sm

We will rely on Proposition~\ref{P:001} and show for compact limiting trees that any limit point satisfies the
strong Markov property in Subsection~\ref{Sub:strongMarkov} and the occupation time formula (\ref{e:abstract})
in Subsection~\ref{Sub:occu}. Note that, if $\treen=(T_n, r_n, \rho_n, \nu_n)$ converges to
$\tree=(T,r,\rho, \nu)$ pointed Gromov-Hausdorff-vaguely (i.e.\ we assume (A1) and (A2) of Theorem~\ref{T:001}),
then compactness of $\tree$ together with assumption (A0) of Theorem~\ref{T:001} is equivalent to the uniform
diameter bound $\sup_{n\in\N} \diam(T_n,r_n) < \infty$.


\subsection{The strong Markov property of the limit}
\label{Sub:strongMarkov}
In this subsection we show that any limit point has the strong Markov property. To be more precise, the main
result is the following:
\begin{proposition}[Strong Markov property]
Let\/ $\tree:=(T,r,\nu)$ and\/ $\tree_n:=(T_n,r_n,\nu_n)$, $n\in\N$, be metric boundedly
finite measure trees.
Assume that all\/ $\tree_n$, $n\in\mathbb{N}$, are discrete with\/ $\sup_{n\in\N} \diam (T_n,r_n) < \infty$, and
that the sequence\/ $(\tree_n)_{n\in\mathbb{N}}$ converges to\/ $\tree$ Gromov-Hausdorff-vaguely as\/ $n\to\infty$.
If\/ $X^n$ is the speed-$\nu_n$ random walk on\/ $(T_n,r_n)$ and\/ $X^n\Tno\tilde{X}$ in path space, then\/
$\tilde{X}$ is a (strong Markov) Feller process.
\label{P:Feller}
\end{proposition}\sm

In order to prove Proposition~\ref{P:Feller}, we will first show that under its assumptions the family of functions
$\{P_n:\,n\in\mathbb{N}\}$, where for each $n\in\mathbb{N}$
\begin{equation}
\label{eq:kernel}
	P_n \colon \left\{\begin{matrix} T_n \times \R_+ &\to& \CM_1(E),\\ (x,t) &\mapsto& \law^x(X^n_t)=:\Pn xt \end{matrix}\right.,
\end{equation}
is uniformly equicontinuous. Here, $\law^x(X^n_t)$ denotes the law of $X^n_t$, where $X^n$ is started in $x\in
T_n$, $E$ is a metric space containing all $T_n$, and $\CM_1(E)$ is equipped with the Prohorov metric.

\begin{lemma}[Equicontinuity]
	Let\/ $\tree:=(T,r,\nu)$ and\/ $\tree_n:=(T_n,r_n,\nu_n)$, $n\in\N$, be metric boundedly finite measure
	trees.
	Assume that all\/ $\tree_n$, $n\in\mathbb{N}$, are discrete with\/ $\sup_{n\in\N} \diam
	(T_n,r_n)<\infty$, and that\/ $\tree_n\to\tree$ Gromov-Hausdorff-vaguely.
	If for each\/ $n\in\mathbb{N}$, $X^n$ is the speed-$\nu_n$ random walk on\/ $(T_n,r_n)$,
	and\/ $P_n\colon\R_+ \times T_n\to\CM_1(E)$ is defined as in (\ref{eq:kernel}), then the family
	$\{P_n:\,n\in\mathbb{N}\}$ is uniformly equicontinuous.
\label{Lem:equicont}
\end{lemma}\sm

\begin{proof}
	Fix $\eps>0$. We construct a $\delta>0$, independent of $n$, such that $\Pn xs$ and $\Pn yt$ are
	$\eps$\nbd close whenever $x,y\in T_n,\, s,t\in \R_+$ are such that $r_n(x,y)<\delta$ and $s\le t \le
	s+\delta$.
	
Fix $n\in\mathbb{N}$, and denote for any two $x,y\in T_n$ by $X^{x}$ and $X^{y}$ speed-$\nu_n$ random walks on
$(T_n,r_n)$ starting in $x$ and $y$, respectively, which are coupled as follows: let the random walks $X^{x}$,
$X^y$ run independently until $X^x$ hits $y$ for the first time, i.e., until $\tau:= \inf\set{t\ge0}{X^x_t=y}$,
and put $X^x_{\tau+\boldsymbol{\cdot}}=X^y$. In particular, whenever $s\ge\tau$, we obtain $X^x_s=X^y_{t-u}$ for $u=\tau+t-s$.

Using the strong Markov property of $X^y$, we can estimate for any $c\in [t-s,t]$
\begin{equation}\label{eq:Sk}
	\bPs{r_n(X^x_s, X^y_t) > \eps} \le \mathbb{P}\{\tau > c-t+s\}
			+ \sup_{z\in T_n} \bPs{\sup_{u\in [0,c]} r_n(z, X^z_u) > \eps}.
\end{equation}
For small $t$, we need another estimate, namely for $r_n(x,y) \le \frac13 \eps$ we have
\begin{equation}\label{eq:smallt}
	\bPs{r_n(X^x_s, X^y_t) > \eps} \le 2 \sup_{z\in T_n} \bPs{\sup_{u\in [0,t]} r_n(z, X^z_u) >
	\tfrac\eps3} =: 2q_t.
\end{equation}
Combining \eqref{eq:Sk} and \eqref{eq:smallt}, we obtain, under the condition $r_n(x,y)\le \frac13\eps$,
for any $c\ge t-s$
\begin{equation}\label{eq:totestim}
	\bPs{r_n(X^x_s, X^y_t) > \eps} \le q_c + q_c \lor \bPs{\tau>c-(t-s)}.
\end{equation}
Note that this estimate depends on $x,y,s,t$ only through $r_n(x,y)$ and $t-s$.

The Gromov-Hausdorff-vague convergence together with the uniform diameter bound on $(T_n, r_n)$ implies that
$(T,r)$ is compact and $(T_n,r_n)$ converges to $(T,r)$ in Gromov-Hausdorff topology. Hence, by
Lemma~\ref{Lem:deg}, $\sup_{n\in\N} \deg_{\frac\eps6}(T_n) < \infty$. Furthermore, the global lower mass-bound
property is satisfied, i.e.\ for every $\eps'>0$, $m_{\eps'}:=\inf_{n\in\N,\,x\in T_n} \nu_n\(B_n(x, \eps')\) > 0$.
We can thus apply Lemma~\ref{Lem:speed} to obtain a sufficiently small $c=c(\eps)>0$, independent of $n$, such that
$q_c \le \frac\eps2$.
To estimate (for this $c$) $\Ps{\tau>c-(t-s)}$, we note that $M:=\sup_{n\in\N} \nu_n(T_n) < \infty$ because of
the diameter bound, and obtain for $t-s \le \frac12 c$
\begin{equation}
\label{meetingtime}
	\bPs{\tau>c- (t-s)} \le \tfrac2c \mathbb{E}[\tau] \le \tfrac4c M \cdot r_n(x,y).
\end{equation}

Choose therefore $\delta:=\frac\eps{8M}c \land \frac\eps3 \land \frac12 c$.
Then for all $x,y\in T_n$ with $r_n(x,y)<\delta$, and $0\le s \le t < s+\delta$, \eqref{eq:totestim} implies
$\bPs{r_n(X^x_s, X^y_t) > \eps} \le \eps$, and hence $d_\mathrm{Pr}(\Pn xs, \Pn yt) \le \eps$,
which is the claimed equicontinuity.
\end{proof}\sm

The proof of Proposition~\ref{P:Feller} relies on the following modification of the Arzel\`a-Ascoli theorem,
which is proven in the same way as the classical theorem.

\begin{lemma}[Arzel\`a-Ascoli]
	Let $(E,d)$ be a compact metric space, $(F, d_F)$ a metric space, $T,T_n\subseteq E$ closed and
	$f_n\colon T_n \to F$ for $n\in\N$. Further assume that the family $\{f_n;\,n\in\N\}$ is uniformly
	equicontinuous with modulus of continuity $h$, and that for all $x\in T$ there exists $x_n\in T_n$ such that
	$x_n \to x$ and $\set{f_n(x_n)}{n\in\N}$ is relatively compact in $F$.
	Then there is a function $f\colon T \to F$, a subsequence of $(f_n)_{n\in\N}$, again denoted
	by $(f_n)$, and $\eps_n>0$ with $\eps_n\to 0$ such that for all $n\in\N$, for all $x\in T$ and $y\in T_n$,
	\begin{equation}
\label{eq:uniform}
		d_F\big(f(x), f_n(y)\big) \le h\big(d(x,y)\big) + \eps_n.
	\end{equation}
\label{Lem:Arzela}
\end{lemma}\sm

Note that \eqref{eq:uniform} in particular implies that $f$ is continuous with the same modulus of continuity
$h$, and that $f_n(x_n) \to f(x)$ whenever $x_n\to x$.

\begin{proof}[Proof of Proposition~\ref{P:Feller}]
By assumption there is a compact metric space $(E,d)$ such that
$T,T_1,T_2,\ldots\subseteq E$, $d\restricted{T}=r$, $d\restricted{T_n}=r_n$ for all $n\in\mathbb{N}$, and
$(T_n,r_n,\nu_n)_{n\in\mathbb{N}}$ converges Hausdorff-weakly to $(T,r,\nu)$.

 According to Proposition~\ref{prop:tight} and Lemma~\ref{Lem:equicont}, the assumptions of
	Arzel\`a-Ascoli are satisfied for the family of functions $P_n$, $n\in\N$, defined in \eqref{eq:kernel}.
	Thus we obtain a continuous subsequential limit $P\colon T \times \R_+ \to \CM_1(E)$, $(x,t)\mapsto P^x_t$.
	Let $S=(S_t)_{t\ge 0}$ and $S^n=(S_t^n)_{t\ge 0}$ be the corresponding operators on $\cont(T)$ and
	$\cont(T_n)$, respectively. That is $S_tf(x):=\inta{T}{f}{P^x_t}$ and
	$S^n_tf(x):=\inta{T_n}{f}{\Pn xt}$, $n\in\mathbb{N}$.
	We show that $S$ is indeed a strongly continuous semigroup.

	To this end, it is enough to show $\lim_{t\to 0} \|S_tf - f\|_\infty = 0$ and $S_{t+s}f=S_t(S_sf)$,
	$s,t>0$, for Lipschitz continuous $f\in \cont(T)$ with Lipschitz constant (at most) $1$ and
	$\|f\|_\infty\le 1$. We can extend every such $f$ to a function on $E$ with the same properties.
	Let $\Lip_1=\Lip_1(E)$ be the space of such (extended) $f$ and recall that the Kantorovich-Rubinshtein
	metric between two measures $\mu,\hat\mu \in \CM_1(E)$,
	\begin{equation} \label{e:KRmetric}
		d_\mathrm{KR}\big(\mu, \hat\mu\big):= \sup_{f\in \Lip_1} \plainint{f}{(\mu-\hat\mu)},
	\end{equation}
	is uniformly equivalent to the Prohorov metric (see \cite[Thm.~8.10.43]{BogachevII}).
	For the rest of the proof, $\CM_1(E)$ is equipped with $d_\mathrm{KR}$.
	Let $h$ be a common modulus of continuity for all $P_n$, $n\in\N$, which exists according to
	\lemref{equicont}. Due to \lemref{Arzela}, $P$ has the same modulus of continuity and hence, for all
	$f\in \Lip_1$,
	\begin{equation}
		\|S_tf - f \|_\infty \le \sup_{x\in T} d_\mathrm{KR}(P_t^x, P_0^x) \le h(t) \xrightarrow[t\to 0]{} 0,
	\end{equation}
	i.e.\ $S$ is strongly continuous.

	Because $T_n$ converges to $T$ in the Hausdorff metric, we find $g_n\colon T_n \to T$ such that
	\begin{equation} \label{e:alphan}
		\alpha_n := \sup_{y\in T_n} d\big(y, g_n(y)\big) \,\tno\, 0.
	\end{equation}
	W.l.o.g.\ we may also assume that $T_1,T_2,\ldots$, are disjoint. As the spaces $(T_n, r_n)$, $n\in\N$,
	are discrete, the map
	\begin{equation} \label{e:024}
		g\colon T\cup \bigcup_{n\in\N} T_n \to T,\quad x \mapsto
			\begin{cases} x, & x\in T\\ g_n(x), & x \in T_n \end{cases}
	\end{equation}
	is continuous. Now we apply \eqref{eq:uniform} to $P_n$ and $P$
	and obtain for all $n\in\N$, $f\in \Lip_1$ and $s>0$
\begin{equation}
\label{e:025}
\begin{aligned}
	\sup_{y\in T_n} \bigl| S_s^nf(y) - (S_sf)\(g(y)\)\bigr|
		&\le \sup_{y\in T_n} d_\mathrm{KR}(\Pn ys, P^{g(y)}_s)\\
		&\le h(\alpha_n) + \eps_n \,\tno\, 0,
\end{aligned}
	\end{equation}
	where $\eps_n$ is obtained in \lemref{Arzela}. For $x\in T$, there exists $x_n\in T_n$ with $x_n\to x$
	and thus, using \eqref{e:025} and the semigroup property of $S^n$,
	\begin{equation} \label{e:026}
\begin{aligned}
	S_{t+s}f(x) &= \nlim S^n_{t+s} f(x_n)
		= \nlim S^n_t(S^n_sf)(x_n) \\
	 &=  \nlim S^n_t(S_sf \circ g)(x_n) = S_t(S_sf \circ g)(x)\\
	 &= S_t(S_sf)(x).
\end{aligned}
\end{equation}
	Now it is standard to see that $S$ comes from a Feller process, and this process has to be $\tilde{X}$.
\end{proof}\sm

We can conclude immediately from Proposition~\ref{P:Feller} that in the general locally compact case any limit
process has the strong Markov property, at least up to the first time it hits the boundary at infinity.

The following example shows that in general we loose the strong Markov property once we hit infinity.
\begin{example}[Entrance law] \label{example2}
	Let $(T,r,\rho)$ be the discrete binary tree with unit edge-lengths, i.e.,
\begin{equation}
\label{e:binary}
   T:=\bigcup_{n\in\N}\{0,1\}^n \cup \{\rho\},
\end{equation}
$r(\rho, x):=n$ for all $x\in \{0,1\}^n$, and there is an edge $x\sim y$ if
	and only if\/ $y = (x,i)$ or $x=(y,i)$ for $i\in\{0,1\}$. 

Put $h(x):=r(\rho,x)$, and
consider the speed measure $\nu(\{x\}) := e^{-h(x)}$, $x\in T$. Obviously, the speed-$\nu$ random walk on $X$ is transient, as $h(X)$ is a reflected random walk on $\N$
with constant drift to the right.

Now consider $(T_n,r,\rho,\nu)$ with $T_n := \bset{x\in T}{h(x)\le n}$, where the metric and the
	measure are understood to be restricted to $T_n$. Because $T_n$ is finite and the speed-$\nu_n$ random walk $X^n$ has no absorbing points, it is positive recurrent. We may therefore conclude from Proposition~\ref{P:007} that for all $x\in T$, $n\in \mathbb{N}$ suitably large,
	\begin{equation}
\label{e:hittingbound}
	\mathbb{E}^x\big[\tau^n_\rho\big]
	=
	2\sum_{y\in T_n} h\big(c(\rho, x, y)\big) e^{-h(y)}
	\le
        \sum_{k=1}^n k2^{k}e^{-k} < \infty.
\end{equation}
Therefore, in contrast to the transience of the speed-$\nu$ random walk on $(T,r)$, any ``limiting'' process $Y$ of
the speed-$\nu_n$ random walks on $(T_n,r_n)$ is also positive recurrent.
This shows that in Theorem~\ref{T:001} we indeed have to stop limiting processes at infinity in order for them
to coincide with the speed-$\nu$ motion on $(T,r)$.
Consequently, this also means that the speed-$\nu$ motion has an entrance law on $(T,r)$ from infinity, which we
obtain by considering excursions of $Y$ away from infinity. Finally, the limit $Y$ obviously looses its strong Markov property at hitting infinity, because, in the
one-point compactification, we are identifying all ends at infinity.
\label{ex:entrance}
\hfill$\qed$
\end{example}\sm


\subsection{The occupation time formula of the limit}
\label{Sub:occu}
In this section we assume that the limiting tree is compact and show that all limit points satisfy the
occupation time formula~(\ref{e:abstract}). The main result is the following:

\begin{proposition}[Occupation time formula]
	Let\/ $\tree:=(T,r,\nu)$ and\/ $\tree_n:=(T_n,r_n,\nu_n)$, $n\in\N$, be  metric boundedly finite measure
	trees.
	Assume that all\/ $\tree_n$, $n\in\mathbb{N}$, are discrete with\/ $\sup_{n\in\N} \diam
	(T_n,r_n)<\infty$, and that\/ $\tree_n\to\tree$ Gromov-Hausdorff-vaguely as $n\to\infty$.
	If\/ $X^n$ is the speed-$\nu_n$ random walk on $(T_n,r_n)$ and\/ $X^n\Tno\tilde{X}$ in path space, then
	$\tilde{X}$ satisfies (\ref{e:abstract}).
\label{P:occupation}
\end{proposition}\sm

To prove this formula, we need a lemma about semi-continuity of hitting times in Skorohod space.
This semi-continuity does not hold in general, but we rather have to use that the limiting path satisfies a
certain regularity property.

If $\supp(\nu)$ is not connected, the paths of the limit process are obviously not continuous. They
satisfy, however, the following weaker {\em closedness condition}.
\begin{definition}[Closed-interval property]
	Let\/ $E$ be a topological space. We say that a function $w\colon \R_+\to E$ has the
	\emph{\clproperty} if\/ $w\([s,t]\)\subseteq E$ is closed for all\/ $0\le s < t$.
\label{Def:closed}
\end{definition}\sm

\begin{lemma}[Speed-$\nu$ motions have the \clproperty]
	The path of the limit process\/ $\tilde{X}$ has the \clproperty, almost surely.
\label{Lem:closedp}
\end{lemma}\sm

\begin{proof}
	Let $A\subseteq T$ be the set of endpoints of edges of $T$. Recall from Remark~\ref{Rem:001} that $A$ is
	at most countable.
	Jumps of the limit process $\tilde{X}$ can only occur over
	edges of $T$, hence $\Xt_{t-} := \lim_{s\nearrow t} \Xt_s \ne \Xt_t$ implies $\Xt_{t-}\in A$.

 Fix $a\in A$. We first show that if  $\tau_a^-:=\inf\{t>0:\,\Xt_{t-}=a\}$ denotes the
 first time when the left limit of $\Xt$ reaches $a$, we have $\tilde{X}_{\tau_a^-} = a$ almost surely, i.e.,
 $\tilde{X}$ does not jump at time $\tau_a^-$ almost surely.
 Indeed, for every $\eps>0$ we can use the right-continuity of the paths of $\tilde{X}$ together with Feller-continuity to
	find $s_0>0$ and $\delta>0$ such that for all $x\in B(a, \delta)$,
	\begin{equation}
		\bPs[x]{\sup_{s\in[0,s_0]}r(a, \tilde{X}_s)>\eps} < \tfrac12\eps.
	\end{equation}
	Define the stopping times $\tau_n := \inf\bset{t\ge 0}{r(\tilde{X}_t, a) \le \frac1n}$, and note that $\tau_n
	\uparrow \tau_a^-$. If $n>\frac1{\delta}$ is such that $\mathbb{P}^x\{\tau_a^--\tau_n>s_0\} <
	\frac12\eps$, then by Proposition~\ref{P:Feller},
	\begin{equation}
		\mathbb{P}^x\big\{r(\tilde{X}_{\tau_a^-}, a) > \eps\big\} \le \tfrac12\eps +
		\mathbb{E}^x\big[\bPs[\Xt_{\tau_n}]{\sup_{s\in[0,s_0]}r(a,\Xt_s) > \eps}\big]
		\le \eps.
	\end{equation}
	Since $\eps$ is arbitrary, this proves $\Xt_{\tau_a^-}=a$ almost surely.

Because $A$ is countable, this implies
	that $\bset{\Xt_u}{u\in[0,t]}$ is closed for all $t\ge 0$, almost surely. Again using the Markov property,
	we also obtain almost surely closedness of $\bset{\Xt_u}{u\in[s,t]}$ for all $t\ge 0,\,s\in\Q_+$, which
	implies closedness for all $s\ge 0$ by right-continuity.
\end{proof}\sm

We omit the proof of the following lemma, because it is straight-forward.
\begin{lemma}[Semi-continuity of the hitting time functional]
	Let\/ $E$ be a Polish space and\/ $\D_E=\D_E(\R_+)$ the corresponding Skorohod space.
	For a set\/ $A\subseteq E$, define
	\begin{equation}\label{eq:sigmaA}
		\sigma_A \colon \D_E \to \R_+\cup\{\infty\},\quad w \mapsto \inf\bset{t\in\R_+}{w(t) \in A}.
	\end{equation}
	Then if\/ $A$ is open, $\sigma_A$ is upper semi-continuous, and if\/ $A$ is closed, the set of lower
	semi-continuity points of\/ $\sigma_A$ contains the set of paths with the \clproperty.
\label{Lem:hitlsc}
\end{lemma}\sm

\begin{remark}
	For\/ $A\subseteq E$ closed, $\sigma_A$ is in general not lower semi-con\-tin\-uous.
\label{Rem:003}
\hfill$\qed$
\end{remark}\sm

\begin{proof}[Proof of Proposition~\ref{P:occupation}]
	Fix $x,y\in T$ and let $\tau_y$ be the first time when $\Xt$ hits $y$.
	It is enough to show \eqref{e:abstract} for non-negative $f\in{\mathcal C}_b(T)$. Because $T$
	is closed in $E$, we can extend $f$ to a bounded continuous function on $E$, again denoted by $f$.
	For $A\subseteq E$, recall the definition of $\sigma_A$ from \eqref{eq:sigmaA} and consider the function
	\begin{equation}
		F_A\colon \D_E \to \R_+\cup\{\infty\}, \quad w\mapsto \integral0{\sigma_A(w)}{f\(w(t)\)}t .
	\end{equation}
	Note that the left-hand side of \eqref{e:abstract} coincides with $\Exp^x[F_y(\Xt)]$, where we
	abbreviate $F_y:=F_{\{y\}}$. The strategy is to approximate $F_y$ by $F_A$ for small neighbourhoods $A$
	of $y$ and then use semi-continuity properties of $F_A$ and the occupation time formula of the
	approximating $X^n$.

	Denote for each $\eps>0$ the closed $\eps$-ball in $E$ around $y$ by $A_\eps$. We claim that almost
	surely
	\begin{equation}\label{eq:tauconv}
		\tau:=\sup_{\eps > 0} \sigma_{A_\eps}(\Xt) = \sigma_{\{y\}}(\Xt) = \tau_y.
	\end{equation}
	Indeed, $\tau \le \tau_y$ is obvious. For the converse inequality, recall that the path of $\tilde{X}$
	almost surely has the \clproperty\ by \lemref{closedp}, which means that $\bset{\Xt_t}{t\in[0,\tau]}$ is
	almost surely a closed set containing points in every $A_\eps$, $\eps>0$, hence also $y$. Therefore
	$\tau_y \le \tau$ almost surely.

	Because $f$ is non-negative, \eqref{eq:tauconv} implies that
	\begin{equation} \label{eq:basconv}
		\sup_{\eps > 0} F_{A_\eps}(\tilde{X}) = F_y(\tilde{X}),
	\end{equation}
	almost surely. Furthermore, it follows from the definition of the Skorohod topology that whenever $w$ is
	a lower- or upper semi-continuity point of $\sigma_A$, the same is true for $F_A$.
	Hence Lemma~\ref{Lem:hitlsc} together with \lemref{closedp} implies that the path of $\Xt$ is almost
	surely a lower semi-continuity point of $F_A$ for closed sets $A$, and an upper semi-continuity point
	for open sets $A$.

	Choose $x_n,y_n\in T_n$ with $y_n\to y$ and $x_n\to x$, and note that $y_n\in A_\eps$ for all sufficiently large $n$.
	Since $X^n\Tno\tilde{X}$, and $\Xt$ is almost surely a lower semi-continuity point of $F_A$,
	\begin{equation}
\label{e:027}
\begin{aligned}
		\mathbb{E}^x\big[F_y(\tilde{X})\big]
  &=
     \sup_{\eps>0} \mathbb{E}^x\big[F_{A_\eps}(\tilde{X})\big]
     \\
       &\le
       \sup_{\eps>0} \nliminf \mathbb{E}^{x_n}\big[ F_{A_\eps}(X^n) \big]
       \\
       &\le \nliminf \mathbb{E}^{x_n}\big[F_{y_n}(X^n)\big] .
 	\end{aligned}\end{equation}

Note that the functions
$(x_n,y_n,z_n) \mapsto 2r_n\(y_n, c_n(x_n, y_n, z_n)\)$ on $T_n^3$
and
$(x,y,z) \mapsto 2r\(y, c(x,y,z)\)$ on $T^3$
have a common Lipschitz continuous extension to $E$ given by
\begin{equation}\label{e:xi}
	\xi(x,y,z) := d(y,x)+d(y,z)-d(z,x).
\end{equation}
Therefore, we obtain from \eqref{e:027} and the occupation time formula for $X^n$ (Proposition~\ref{P:007}) that
\begin{equation}
\label{e:028}
\begin{aligned}
		\mathbb{E}^x\big[F_y(\tilde{X})\big]
  &\le
    \nliminf \int\nu_n(\mathrm{d}z)\,\xi(x_n, y_n, z)f(z)
    \\
  &=  2\int\nu(\mathrm{d}z)\,r\big(y,c(x,y,z)\big)f(z).
	\end{aligned}\end{equation}

	On the other hand, for every sufficiently small $\eps>0$ and large $n\in\N$, there is a unique point
	$y_n'\in B(y_n, 2\eps)\cap T_n$ closest to $x_n$, and using that $\Xt$ is almost surely an upper
	semi-continuity point of $F_{B(y, \eps)}$, we obtain
	\begin{equation}
\label{e:029}
\begin{aligned}
	\mathbb{E}^x\big[F_y(\tilde{X})\big] &\ge \nlimsup \mathbb{E}^{x_n}\big[F_{B(y,\eps)}(X^n)\big] \\
		&\ge \nlimsup \mathbb{E}^{x_n}\big[F_{B(y_n,2\eps)}(X^n)\big]\\
		&= \nlimsup \mathbb{E}^{x_n}\big[F_{y_n'}(X^n)\big]\\
		&\ge  2\int\nu(\mathrm{d}z)\,\big(r(y,c(x,y,z)\big)-2\eps\big) f(z).
\end{aligned}\end{equation}
	The claim follows with $\eps\to 0$.
\end{proof}\sm

\section{Proof of Theorem~\ref{T:001}}
\label{S:proofT001}
In this section, we collect all the pieces we have proven so far and present the proof of our invariance principle.

As we have stated all the results which characterize the limiting process for approximating rooted metric
measure trees $(T_n,r_n,\rho_n,\nu_n)$ where $(T_n,r_n)$ was assumed to be discrete,
we start with a lemma which states that each rooted metric boundedly finite measure tree can be approximated by discrete trees.
\begin{lemma}[Approximation by discrete trees] Let\/ $(T,r,\rho,\nu)$ be a rooted metric boundedly finite
	measure tree\/ $\tree$. Then we can find a sequence\/ $\tree_n:=(T_n,r_n,\rho,\nu_n)$ of rooted discrete
	metric boundedly finite measure trees such that\/ $\tree_n\to\tree$ pointed Gromov-Hausdorff-vaguely.
\label{L:discrete}
\end{lemma}\sm

\begin{proof} Let $(T,r,\rho,\nu)$ be a rooted metric boundedly finite measure tree, and for each $n\in\mathbb{N}$,
$S_n$ a finite $\frac{1}{n}$-net of $B(\rho,n)$ containing $\{\rho\}$. Let $T_n\subseteq T$ be the smallest
metric tree containing $S_n$, i.e.\ the union of $S_n$ and all branching points $x\in T$ with
\begin{equation}
\label{e:branch}
  r(x,s_1)=\tfrac{1}{2}\big(r(s_1,s_2)+r(s_1,s_3)-r(s_2,s_3)\big),
\end{equation}
for some $s_1,s_2,s_3\in S_n$. As usual, let $r_n$ be the restriction of $r$ to $T_n$, and note that $T_n$ is a
finite set, hence $(T_n, r_n)$ is a discrete metric tree.

Consider for each $n\in\mathbb{N}$ the map $\psi_n:T\to T_n$ which sends a point in $T$ to the nearest point on the way from $x$ to $\rho$ which belongs to $T_n$, i.e.,
\begin{equation}
\label{e:psin}
   \psi_n(x):= \sup\big\{y\in T_n:\, y\in[\rho, x]\big\}.
\end{equation}
Finally, put
\begin{equation}
\label{e:nun}
   \nu_n:=\big(\psi_{n}\big)_\ast\nu\restricted{B(\rho, n)}.
\end{equation}

Then, obviously, the Prohorov distance between $\nu\restricted{B(\rho, n)}$ and $\nu_n$ is not larger than $\tfrac{1}{n}$. Thus
$(T_n,r_n,\rho,\nu_n)$ converges pointed Gromov-vaguely and also pointed Gromov-Hausdorff-vaguely to $(T,r,\rho,\nu)$.
\end{proof}\sm

\subsection{Compact limit trees}
\label{Sub:compact}
In this subsection we restrict to the case where the limiting tree is compact. We start with the proof of
Proposition~\ref{P:001}, on which we shall rely the characterization of the limit process.

\begin{proof}[Proof of Proposition~\ref{P:001}]
	Consider $X$ and $Y$ satisfying the assumption on Proposition~\ref{P:001}.
	In particular, assume that $X_{\boldsymbol{\cdot}\wedge\tau_y}$ is transient for all $y\in T$.
	Consider for each $y\in T$ the family of resolvent operators $\{G^{X,y}_\alpha;\,\alpha>0\}$ and
	$\{G^{Y,y}_\alpha\;\alpha> 0\}$ associated with $\{X_{\boldsymbol{\cdot}\wedge\tau_y};\,y\in T\}$ and
	$Y_{\boldsymbol{\cdot}\wedge\tau_y}$, and put $G_X^y:=\lim_{N\to\infty}G^{X,y}_{1/N}$ and
	$G_Y^y:=\lim_{N\to\infty}G^{Y,y}_{1/N}$, respectively. By transience, $G_X^y<\infty$ for all $y\in T$.
	Moreover, for all $x\in T$, and bounded, measurable $f\colon T\to\R_+$, 
	\begin{equation}
	\label{e:corres}
	   G_X^yf(x)=\mathbb{E}^x\big[\int^{\tau_y}_0\mathrm{d}s\,f(X_s)\big].
	\end{equation}
	By (\ref{eq:occueqal}), $G_Y^yf(x)<\infty$ as well.

	As $X$ is a strong Markov processes, the resolvent identity holds, i.e.,
	\begin{equation}
	\label{e:Green}
	   G^{X,y}_\alpha =  G^{X,y}_\beta  +(\alpha -\beta) G^{X,y}_\alpha G^{X,y}_\beta.
	\end{equation}
	Iterating the latter with $\alpha > \beta >0$ and $|\alpha - \beta| \leq \frac{1}{2\| G^{X,y}_\beta \|} $, we have
	\begin{equation}
	\label{e:022}
	   G^{X,y}_\alpha =  G^{X,y}_\beta  +(\alpha -\beta) \big(G^{X,y}_\beta\big)^2+(\alpha -\beta)^2
	   \big(G^{X,y}_\beta\big)^3+\cdots
	\end{equation}

	We note that $ \|G^{X,y}_\beta \| \leq \|G^{X,y} \| $ for all $\beta \geq 0$. So it is bounded above
	independent of $\beta$. Hence (\ref{e:022}) holds for $\beta =0$ by taking limits.
	Further, by the same arguments, \eqref{e:022} also holds for $Y$ instead of $X$, and by
	\eqref{eq:occueqal} $G^{Y,y}_0:=G_Y^y= G_X^y$. 	
	Therefore, for all small enough $\alpha>0$, $G^{X,y}_\alpha =
	G^{Y,y}_\alpha$.  Thus for all small enough $\alpha >0$,
	\begin{equation}
	\label{eq:occueqalpha}
		\mathbb{E}^x\big[\int_0^{\tau_y}\mathrm{d}t\,e^{-\alpha t}\cdot f(X_t)\big]
		=\mathbb{E}^x\big[\int_0^{\tau_y}\mathrm{d}t\,e^{-\alpha t}\cdot f(Y_t)\big].
	\end{equation}

	Therefore by uniqueness of the Laplace transform,
	\begin{equation}
	\label{eq:occueqlaplace}
			\mathbb{E}^x\big[f(X_t);\,\{t<\tau_y\}\big]
			=\mathbb{E}^x\big[f(Y_t);\,\{t<\tau_y\}\big]
		\end{equation}
	for all $y\in T$ and for all $t>0$. Therefore the one dimensional distributions of $X_{\boldsymbol{\cdot} \wedge
	\tau_y}$ and $Y_{\boldsymbol{\cdot}\wedge \tau_y}$ are the same for all $y\in T$.
	By the strong Markov property, this implies that the laws of $X$ and $Y$ agree.
\end{proof}\sm

To show f.d.d.\ convergence, we need to control the probability that $X_t$ is in an ``exceptional'' set of small
$\nu$\nbd measure. To this end, we use the following simple heat-kernel bound. We will see in
\corref{heatkernel} below that the technical assumption $\nu(\{x\}) > 0$ can be dropped.

\begin{lemma}\label{Lem:heatkernel}
	Let\/ $\tree:=(T,r,\nu)$ be a compact metric finite measure tree, $x\in T$ with $\nu(\{x\})>0$, and\/ $X$
	the speed\nbd $\nu$ motion on\/ $(T,r)$ started in\/ $x$.
	Then the law of\/ $X_t$ has for every\/ $t>0$ a density\/ $q_t(x,\cdot)\in L^2(\nu)$ w.r.t.\ $\nu$, and
	\begin{equation}\label{eq:heatkernel}
		\bigl\|q_t(x,\cdot)\bigr\|_2^2 \,\le\, \nu(T)^{-1} + \diam(T)\cdot t^{-1}
				\qquad \forall t>0,
	\end{equation}
	where\/ $\|\cdot\|_2$ is the norm in\/ $L^2(\nu)$.
	In particular, for any\/ $A\subseteq T$, we have
	\begin{equation}\label{eq:Abound}
		\Ps[x]{X_t\in A} \le \gamma_t \sqrt{\nu(A)} \qquad \forall t>0,
	\end{equation}
	where the constant\/ $\gamma_t:=1+\nu(T)^{-1} + \diam(T)\cdot t^{-1}$ is independent of\/ $x$ and depends
	on\/ $(T,r,\nu)$ only through\/ $\nu(T)$ and\/ $\diam(T)$.
\end{lemma}
\begin{proof}
\emph{1.} Let $f:=\nu(\{x\})^{-1} \indicator{\{x\}}$ be the density of $\delta_x$ w.r.t.\ $\nu$, and
	\begin{equation}
		f_t:=P_tf,\qquad g(t) := \|f_t\|_2^2,
	\end{equation}
	where $(P_t)_{t\ge 0}$ is the semi-group of the speed\nbd $\nu$ motion.
	Due to reversibility of $\nu$ it is easy to see that $f_t=q_t(x,\cdot)$ is the density of $X_t$ w.r.t.\
	$\nu$. Furthermore,
	\begin{equation}\label{eq:gprime}
		g'(t) = 2\langle Gf_t, f_t \rangle_\nu = -2\E(f_t,f_t),
	\end{equation}
	where $G$ is the generator of $(P_t)_{t\ge 0}$. Let $a:=\diam(T)^{-1}$.
	Because $\|f_t\|_1=1$, we find a point $y\in T$ with $f_t(y) \le b:=\nu(T)^{-1}$.
	For every $z\in T$ with $f_t(z)\ge b$, we have
	\begin{equation}\label{eq:Ebound}
		\E(f_t,f_t) \ge \(f_t(z) -f_t(y)\)^2\cdot \(2r(z,y)\)^{-1} \ge \tfrac12 a\(f_t(z) - b\)^2.
	\end{equation}
	Combining \eqref{eq:Ebound} and \eqref{eq:gprime}, and using $g(t)=\|f_t\|_2^2\le
	\|f_t\|_\infty\|f_t\|_1 = \|f_t\|_\infty$, we obtain the differential inequality
	\begin{equation}
		g'(t) \le - a (\|f_t\|_\infty - b)^2 \le -a\(g(t) - b\)^2.
	\end{equation}
	In the above, we have used that $g(t) \ge b$.
	Solving $h'_u(t)=-a(h_u(t)-b)^2$, $h_u(0)=u$, and using monotonicity of the solution in $u$, we conclude
	\begin{equation}
		g(t) \le \lim_{u\to\infty} h_u(t) = \lim_{u\to\infty} \frac{u(1+abt) - ab^2t}{uat-bat+1}
			= b+(at)^{-1},
	\end{equation}
	which is the desired bound \eqref{eq:heatkernel}.

\proofstep{2.}
	For $u:=\nu(A)^{-1/2}$ we obtain
	\begin{equation}
		\Ps[x]{X_t\in A} \le  u\,\nu(A) + \inta{\{f_t>u\}}{\frac{f_t^2}u}\nu
			\le \sqrt{\nu(A)}\(1 + \|f_t\|^2_2).
	\end{equation}
	Together with \eqref{eq:heatkernel} this implies the desired bound \eqref{eq:Abound}.
\end{proof}

\begin{proposition}[Theorem~\ref{T:001} holds for compact limit trees]\label{P:compact}
Let\/ $\tree:=(T,r,\rho,\nu)$, $\tree_1:=(T_1,r_1,\rho_1,\nu_1)$, $\tree_2:=(T_2,r_2,\rho_2,\nu_2), \ldots$ be
rooted metric boundedly finite measure trees with\/ $\sup_{n\in\N} \diam (T_n,r_n)<\infty$.
Let\/ $X$ be the speed-$\nu$ motion on\/ $(T,r)$ starting in $\rho$, and for all\/ $n\in\mathbb{N}$,
$X^n$ the speed\nbd $\nu_n$ motion on\/ $(T_n,r_n)$ started in $\rho_n$.
Assume that the following conditions hold:
\begin{itemize}
	\item[(A1)] The sequence $(\tree_n)_{n\in\mathbb{N}}$ converges to $\tree$
		pointed Gromov-vaguely.
	\item[(A2)] The uniform local lower mass-bound property (\ref{e:assume2}) holds.
\end{itemize}
Then the following hold:
\begin{enumerate}
	\item $X^n$ converges weakly in path-space to\/ $X$.
	\item If we assume only (A1) but not (A2), then\/  $X^n$ converges in finite dimensional distributions
	to\/ $X$.
\end{enumerate}
\end{proposition}\sm

\begin{proof} Assume w.l.o.g.\ that $(\tree_n)_{n\in\mathbb{N}}$ are discrete trees (the general result is then
obtained by Lemma~\ref{L:discrete} and a diagonal argument).
Let $X^n$ be a sequence of $\nu_n$-random walks on $(T_n,r_n)$ starting in $\rho_n$.

\sm(i)\, By Proposition \ref{prop:tight} we know that the sequence is tight. Let $\tilde{X}$ be a weak subsequential
limit on $(T,r)$.
Then in particular, $\tilde{X}_0=\rho$ almost surely.
From \propref{Feller} together with \propref{occupation} we know that $\tilde{X}$ is a strong Markov process and
$\mathbb{E}^x[\int^{\tau_z}_0\mathrm{d}s\,f(\tilde{X}_s)]= 2\int\nu(\mathrm{d}y)\,r(z, c(x,y,z))f(y).$

Let $X$ be the speed-$\nu$ motion on $(T,r)$ starting in $\rho$. Then
$X$ is the strong Markov process associated with the Dirichlet form $({\mathcal E},{\mathcal D}({\mathcal E}))$.
$X$ is recurrent as clearly $\mathbf{1}\in{\mathcal D}({\mathcal E})$ and ${\mathcal E}(\mathbf{1},\mathbf{1})=0$.
Thus $X$ satisfies (\ref{e:abstract}) by Proposition~\ref{P:007}.
Moreover, it follows from
Lemma~\ref{Lem:001} that $X_{\boldsymbol{\cdot}\wedge\tau_y}$ is transient for all $y\in T$.
Therefore the laws of $\tilde{X}$ and $X$ agree by Proposition~\ref{P:001}. \sm

\sm(ii)\, Using that $\treen$ converges Gromov-weakly to $\tree$, and $\tree$ is compact, we can construct subsets
$A_n \subseteq T_n$ with $\nu_n(A_n) \to 0$, $\rho_n\not\in A_n$ and the following property. The measure trees
$\smallxt_n := (\Tt_n, r_n, \rho_n, \nu_n)$, where $\Tt_n := T_n \setminus A_n$, satisfy the lower mass-bound
\eqref{e:assume2} and still converge Gromov-weakly to $\tree$.
Let $\Xt^n$ be the $\nu_n$-random walk on $(\Tt_n, r_n)$. Then $\Xt^n$ converges in distribution to
$X$ by part (i). We show that every finite-dimensional marginal of $\Xt^n$ is weakly merging with the
corresponding marginal of $X^n$. For this it is enough to show for all $t\ge 0$ the uniform merging of
one-dimensional marginals, i.e.\
\begin{equation}\label{e:ProhorovMerging}
	\nlim \sup_{x\in \Tt_n} d_\mathrm{Pr}^{(T_n,r_n)}\(\law^x(X^n_t),\, \law^x(\Xt^n_t)\) = 0,
\end{equation}
where $d_\mathrm{Pr}^{(T_n,r_n)}$ is the Prohorov metric associated to $r_n$. The finite-dimensional
statement then follows from the Markov property of the speed-$\nu$ motions together with the Feller continuity
of the limiting process (proven in \propref{Feller}).

Recall that $(T_n, r_n)$ is discrete and thus $\nu_n(\{x\})>0$ for all $x\in T_n$. Using \lemref{heatkernel},
and the fact that $\diam(T_n)$ and $\nu_n(T_n)^{-1}$ are bounded uniformly in $n$, we obtain $\gamma_t>0$,
independent of $n$, such that
\begin{equation}\label{eq:exbound}
	\sup_{x\in T_n}\Ps[x]{X_t^n \in A_n} \le \gamma_t\sqrt{\nu_n(A_n)}.
\end{equation}
We can couple $X^n$ and $\Xt^n$ by a time transformation such that $\Xt^n_t = X^n_{L_n^{-1}(t)}$, where
$L_n^{-1}(t) = \inf\set{s\ge 0}{\integral0s{\indicator{\Tt_n}(X^n_u)}u > t}$. For \eqref{e:ProhorovMerging} it
is enough to show for every fixed $t,\eps>0$ that
\begin{equation}\label{eq:merging}
	\sup_{x\in\Tt_n} \bPs[x]{r_n(X_t^n, \Xt^n_t) > \eps} \le 4\eps,
\end{equation}
for all sufficiently large $n\in\N$. The idea is that $X_t^n$ and $\Xt^n_t$ do not differ too much, because
$\Xt^n_t$ cannot move far in a short amount of time and will be ahead of $X_t^n$ only a small amount of time,
controlled via the occupation time formula by the (small) $\nu_n$-measure of $A_n=T_n\setminus \Tt_n$.

Because $\smallxt_n$ converges Gromov-Hausdorff weakly, we can use the speed bound, \lemref{speed}, to find
$c>0$ such that the probability that $\Xt^n$ moves $\eps$ within time $c$ is bounded by $\eps$, i.e.,
\begin{equation}\label{eq:speedestim}
	\sup_{x\in \Tt_n} \bPs[x]{\sup_{s\in [0,c]}r_n(\Xt^n_s, x) > \eps} \le \eps.
\end{equation}
In order to use the occupation time formula, we fix two points $y_n,z_n \in \Tt_n$ with $r_n(y_n,z_n) > \eps$
and define recursively the times where $X^n$ hits $y_n$ and $z_n$ in alternation, i.e.\ $\tau_n^0:=0$,
$\tau_n^k := \inf\set{t>\tau_n^{k-1}}{X^n_t=y_n}$ for $k$ odd and $\tau_n^k :=
\inf\set{t>\tau_n^{k-1}}{X^n_t=z_n}$ for $k$ even.
Let $\tilde\tau_n^k$, $k\in\N$, be the analogous stopping times for $\Xt^n$ instead of $X^n$.
Because the lower bound for the distance of $y_n$ and $z_n$ is independent of $n$,
we can use \lemref{speed} again to find $k\in\N$, independent of $n$, such that
$\Ps{\tilde\tau_n^k < t} < \eps$. Because $\tau_n^k \ge \tilde\tau_n^k$, we also obtain
\begin{equation}\label{eq:tauestim}
	\sup_{x\in \Tt_n} \Ps[x]{\tau_n^k < t} < \eps.
\end{equation}

Now consider the accumulated time difference between $X^n$ and $\Xt^n$ until $\tau_n^k$, i.e.,
\begin{equation}
	\delta_n := \integral0{\tau_n^k}{\indicator{A_n}(X^n_t)}t.
\end{equation}
Then, by the occupation time formula,
\begin{equation}
	\sup_{x\in \Tt_n} \Exp^x[ \delta_n ] \le k \cdot 2\diam(T_n) \nu_n(A_n).
\end{equation}
The right-hand side tends to zero as $n$ tends to infinity, because $\diam(T_n)$ is uniformly bounded by
assumption and $k$ is independent of $n$. Therefore, for sufficiently large $n$ depending on $c$ chosen in
\eqref{eq:speedestim},
\begin{equation}\label{eq:deltaestim}
	\sup_{x\in \Tt_n} \Ps[x]{\delta_n > c} < \eps.
\end{equation}
On the event $\{X_t^n \not\in A_n\}$, we have $X_t^n=\Xt^n_{L_n(t)}$, and on the event $\{\tau_n^k \ge t\}$, we
have $t-L_n(t) < \delta_n$. Hence, using \eqref{eq:exbound} and \eqref{eq:tauestim}, we obtain for all $x\in
\Tt_n$,
\begin{equation}
\begin{aligned}
	&\bPs[x]{r_n(X_t^n, \Xt^n_t) > \eps} \\
		&\;\le \Ps[x]{X_t^n \in A_n} + \Ps[x]{\tau_n^k< t}
			+ \bPs[x]{t-L_n(t) < \delta_n,\,r_n(\Xt^n_{L_n(t)}, \Xt^n_t) > \eps}\\
		&\;\le \gamma_t\sqrt{\nu_n(A_n)} + \eps + \Ps[x]{\delta_n > c}
			+ \bPs[x]{\sup_{s\in [t-c, t]} r_n(\Xt^n_s, \Xt^n_t)},
\end{aligned}
\end{equation}
which is bounded by $4\eps$ for large $n$ due to  $\nu_n(A_n) \to 0$, \eqref{eq:deltaestim} and
\eqref{eq:speedestim} together with the Markov property of $\Xt^n$.
This proves \eqref{eq:merging} and hence the claimed f.d.d.\ convergence.
\end{proof}\sm

\begin{cor}[pointwise $L^2$-heat-kernel bound]\label{cor:heatkernel}
	\lemref{heatkernel} remains correct if we drop the assumption\/ $\nu(\{x\})>0$. In particular, for
	every compact metric finite measure tree\/ $\tree:=(T,r,\nu)$, the following bound on the\/ $L^2(\nu)$-norm
	of the heat-kernel\/ $q_t$ (defined in \lemref{heatkernel}) holds:
	\begin{equation}
		\bigl\|q_t(x,\cdot)\bigr\|_2^2 \,\le\, \nu(T)^{-1} + \diam(T)\cdot t^{-1}
				\qquad \forall x\in T,\, t>0.
	\end{equation}
\end{cor}
\begin{proof}
	Fix $x\in T$, $t>0$, and let $\nu_n:=\nu+\frac1n\delta_x$. Let $X^n$ and $X$ be the speed\nbd $\nu_n$ and
	speed\nbd $\nu$ motion on $(T,r)$, respectively, all started in $x$. According to \propref{compact} for
	$\treen:=(T_n, r_n, \rho_n, \nu_n):= (T, r, x, \nu + \frac1n \delta_x)$, the law $\mu_{n,t}$ of $X_t^n$
	converges weakly to the law $\mu_t$ of $X_t$.
	According to \lemref{heatkernel}, there is $f_{n,t}\in L^2(\nu)$ with $\mu_{n,t} = f_{n,t}\cdot \nu$, and
	$\|f_{n,t}\|_2$ is bounded uniformly in $n$. Therefore, the weak limit $\mu_t$ also admits a density
	with the same bound on its $L^2(\nu)$-norm.
\end{proof}

We conclude this subsection with examples showing how the violation of the tightness condition (A2) destroys
convergence in path space, while f.d.d.\ convergence still holds.
\begin{example}[f.d.d.\ convergence but not path-wise]
Let\/ $r,r_1,r_2,\ldots$ be the Euclidean metric on\/ $[0,1]$.
\begin{enumerate}
	\item Let\/ $T_n=\{0,1\}$, and\/ $\nu_n=\delta_0 + \tfrac1n \delta_1$ for
		$n\in\N$. Then $\treen := (T_n, r, 0, \nu_n)$ converges pointed Gromov-vaguely to $\tree:=(\{0\}, r, 0,
		\delta_0)$. The speed-$\nu_n$ motion $X^n$ is a two-state Markov chain that jumps from $0$ to
		$1$ at rate $\half$ and from $1$ to $0$ at rate $\frac{n}{2}$. It obviously converges f.d.d.\ to
		the constant process, but not in path-space.
	\item  Let $T_n=[0,1]$, and $\nu_n=\delta_0 + \delta_1 + \frac1n \lambda_{[0,1]}$, where
		$\lambda_{[0,1]}$ is Lebesgue measure on $[0,1]$. Then $(T_n, r, 0, \nu_n)$ converges
		pointed Gromov-vaguely to $\(\{0,1\}, r, 0, \nu\)$ with $\nu=\delta_0+\delta_1$. The speed-$\nu$ motion
		$X$ is the symmetric Markov chain on $\{0,1\}$ with jump-rate $\half$, and the speed-$\nu_n$
		motions $X^n$ are sticky Brownian motions on $[0,1]$ with diverging speed on $(0,1)$, as $n$
		tends to $\infty$.
		As $X^n$ has continuous paths for each $n\in\mathbb{N}$ but   $X$ has discontinuous paths, the convergence cannot be in path space.
        The finite dimensional distributions of $X^n$, however, converge to those of $X$, as the processes $X^n$ spend less and less times in discontinuity points.
        \hfill$\qed$
\end{enumerate}
\label{example3}
\end{example}\sm

\subsection{From compact to locally compact limit trees}\label{Sub:(loc)compact}
In this subsection we extend the proof of Theorem~\ref{T:001} to locally compact trees equipped with boundedly
finite speed measures.  In order to reduce this to the compact case, we stop the processes upon reaching a
height $R$. For that purpose we need the following lemma whose proof is straight-forward and will therefore be
omitted.

Recall the closed interval property from Definition~\ref{Def:closed}.

\begin{lemma}[Continuity points]
	Let $(E, d)$ be a Polish space, $\rho \in E$, and $R> 0$. Define the function
	\begin{equation}
		\psi_R\colon \D_E \to \D_E,\quad
		\psi_R(w)(t) := w\big(t\land \inf\bset{s}{d(\rho,w(s))\ge R}\big).
	\end{equation}
	Assume that $w\in\D_E$ has the \clproperty, and that the map $t\mapsto d\(\rho,
	w(t)\)$ does not have a local maximum at height\/ $R$. Then $w$ is a continuity point of\/ $\psi_R$.
\label{Lem:stopcont}
\end{lemma}\sm

\begin{proof}[Proof of Theorem~\ref{T:001}]
(ii) has already been shown in Proposition~\ref{P:compact}.

(i) We call a point $v\in T$ {\em extremal leaf} of $T$ if the height function $h\colon T\to \R_+$, $x\mapsto r(\rho, x)$ has a local maximum at $v$. Note that,
although there can be uncountably many extremal leaves, the
set of heights of extremal leaves is at most countable due to separability of $T$.
Now choose $R_k>0$, $k\in\N$, with $R_k\to\infty$ such that there is no extremal leaf of $T$ at height $R_k$ and
$\nu\big\{x'\in T:\,r(\rho,x')=R_k\big\}=0$.

Let $X$ be the speed-$\nu$ motion on $(T,r)$ started in $\rho$, and recall that $X=X_{\boldsymbol{\cdot}\wedge\zeta}$, where $\zeta:=\inf\{t\ge 0:\,r(\rho,X_t)=\infty\}$.
We show that the law of $X$
coincides with the law of $\tilde{X}_{\boldsymbol{\cdot}\wedge\zeta}:=\psi_\infty(\tilde{X})$, where $\tilde{X}$
is any limit process. Using that there is no extremal leaf of
$T$ at height $R_k$ and that $\tilde{X}$ and $X$ have the \clproperty, we obtain from \lemref{stopcont} that (the paths
of) $\tilde{X}$ and $X$ are almost surely continuity points of $\psi_{R_k}$.

Let $X^n_k$ be the speed-$\nu_n$ motion on the compact metric measure tree $T_n\restricted{B(\rho_n, R_k)}$ and $X_k$
the speed-$\nu$ motion on the compact metric measure tree $T\restricted{B(\rho, R_k)}$.
Then, for every $k\in\N$, $X^n_k\Tno X_k$, as $n\to
\infty$, by \propref{compact}. Furthermore, for every $k$ there is an $\ell=\ell_k$, such that the laws of
$\psi_{R_k}(X^n)$ and $\psi_{R_k}(X^n_\ell)$ coincide; and the same is true for $\psi_{R_k}(X)$ and
$\psi_{R_k}(X_\ell)$.

By continuity of $\psi_{R_k}$ in $\tilde{X}$ and $X$, we obtain
\begin{equation}
\label{e:030}
	\psi_{R_k}\big(X^n_\ell\big) \eqlaw \psi_{R_k}\big(X^n\big) \convlaw \psi_{R_k}\big(\tilde{X}\big),
\end{equation}
and on the other hand
\begin{equation}
	\psi_{R_k}\big(X^n_\ell\big) \convlaw \psi_{R_k}\big(X_\ell\big) \eqlaw \psi_{R_k}\big(X\big).
\end{equation}
Hence $\psi_{R_k}(\tilde{X}) \eqlaw \psi_{R_k}(X)$ for all $k\in\mathbb{N}$, and therefore $\psi_\infty(\tilde{X}) \eqlaw \psi_\infty(X) = X$
as claimed.
\end{proof}\sm

\section{Examples and related work}
\label{S:example}
We conclude the paper with a discussion on how our invariance principle relates to results from the existing literature.
These results have often been proven
via quite different techniques but they all follow in a unified way  from Theorem \ref{T:001}.

In Subsection~\ref{Sub:Stone} we revisit \cite{Stone1963} which (including a killing part) proves the invariance
principle in the particular situation when the underlying metric trees are closed subsets of $\R$, or
equivalently, linear trees.
In Subsection~\ref{Sub:Kigami} we connect our invariance principle with the construction of diffusions on
so-called dendrites, or equivalently, $\R$-trees, which is given in \cite{Kigami95}.
We continue in Subsection~\ref{Sub:Croydon} with \cite{Croydon2010}, where the classical convergence of
rescaled simple random walks on $\Z$ to Brownian motion on $\R$ is generalized in a different
direction than in \cite{Stone1963}. Namely, simple random walks on discrete trees with uniform edge-lengths are
proven to converge to Brownian motion on a limiting rooted compact $\R$-tree which additionally has to satisfy
some conditions.
Finally, in Subsection~\ref{Sub:Kesten} we consider the nearest neighbor random walk on a size-biased branching
tree for which the suitably rescaled height process averaged over all realizations is tight according
to~\cite{MR88b:60232}, while for almost every fixed realization it is not tight by \cite{BarlowKumagai2006}.

\subsection{Invariance principle on $\R$}
\label{Sub:Stone}

In this subsection, we consider the special case of linear trees, i.e., closed subsets of $\R$.

Let\/ $\nu,\, \nu_n$, $n\in\N$, be locally finite measures on $\R$, $T:=\supp(\nu)$ and $T_n:=\supp(\nu_n)$.
Denote the Euclidean metric on $\R$ by $r$. Then $(T,r,0, \nu)$ and $(T_n,r,0,\nu_n)$ are obviously rooted
metric boundedly finite measure trees in the sense of Definition~\ref{Def:002}. Also note that the speed\nbd
$\nu$ motion is conservative (i.e.\ does not hit infinity), because the tree $(T,r)$ is recurrent (see, e.g.,
\cite[Theorem~4]{AthreyaEckhoffWinter2013}).
Now if $\nu_n$ converges vaguely to $\nu$, and the uniform local lower mass-bound \eqref{e:assume2} holds,
Theorem~\ref{T:001} implies that the speed-$\nu_n$ motions converge in path-space to the speed-$\nu$ motion.
This (essentially) is Theorem~1 (i) obtained in \cite{Stone1963} in the special case, where the killing
measures are not present.

The methods used in \cite{Stone1963} are quite different from ours. In that paper  all processes are represented as
time-changes of standard Brownian motion and  a jointly continuous version of local times is used.

\begin{example}[Standard motion on disconnected sets] A particular instance of Stone's invariance principle was
studied in detail in \cite{BhamidiEvansPeledRalph2008}. Put for each $q>1$, $T_q:=\{\pm
q^k;\,k\in\mathbb{Z}\}\cup\{0\}$ and $\rho_q = 0$. Then $(T_q)_{q>1}$ converges, as $q\downarrow 1$, to $\mathbb{R}$
with respect to the localized Hausdorff distance. Recall the length measure from (\ref{length}). Obviously, as the length
measure is always boundedly finite on linear trees, the embedding which sends a rooted tree $(T,\rho)$ with
$T\subseteq\mathbb{R}$ to the measure tree $(T,\rho,\lambda^{(T,\rho)})$ is a homeomorphism onto its image. Thus
$(T_q,0,\lambda^{(T_q,0)})$ converges Hausdorff-vaguely to $(\mathbb{R},0,\lambda)$, as $q\downarrow 1$, where
$\lambda$ is the Lebesgue measure. It therefore follows that the speed-$\lambda^{(T_q,0)}$ motion on $T_q$
converges in path space to the standard Brownian motion on $\mathbb{R}$
by Theorem~\ref{T:001}. The latter is Proposition~5.1 in \cite{BhamidiEvansPeledRalph2008}.
 \label{Exp:050}
 \hfill$\qed$
\end{example}\sm

\subsection{Diffusions on dendrites}
\label{Sub:Kigami}
In \cite{Kigami95}
diffusions on dendrites (which are $\R$-trees) are constructed via approximating Dirichlet forms rather than processes.
In this subsection we relate our invariance principle to this construction.

Let $(T,r,\rho,\nu)$ be a complete, locally compact, rooted boundedly finite measure $\R$-tree. Let furthermore $(T_m)_{m\in\mathbb{N}}$ be an increasing
family of finite subsets of $T$.
Put for all $f,g:T_m\to\R$
\begin{equation}
\label{Kigami:m}
\begin{aligned}
   {\mathcal E}_m\big(f,g\big)
  &
   :=
   \tfrac{1}{2}\int_{T_m}\lambda^{(T_m,r_m,\rho)}(\mathrm{d}y)\,\nabla f(y)\nabla g(y).
\end{aligned}
\end{equation}
Assume for each $m\in\mathbb{N}$ that $T_m$ contains all the branch points of the subtree spanned by $T_m$ (see our condition (\ref{e:br})). Then for all $m\le m'$,
and for all $f:T_m\to\R$,
\begin{equation}
\label{Kigami:mm}
\begin{aligned}
   {\mathcal E}_m\big(f,f\big)
 &=
   \min\big\{{\mathcal E}_{m'}\big(g,g\big):\,g:T_{m'}\to\mathbb{R},\,g\restricted{T_m}=f\big\}.
\end{aligned}
\end{equation}
That is, the sequence $(T_m, \E_m)_{m\in\mathbb{N}}$ is compatible in the sense of Definition~0.2 (and the
following paragraph) in \cite{Kigami95}. Assume further that $T^\ast:=\cup_{m\in\mathbb{N}}T_m$ is dense in $T$,
and consider the bilinear form
\begin{equation}
\label{Kigami:form}
   {\mathcal E}^{\mathrm{Kigami}}(f,g)
 :=
   \lim_{m\to\infty}{\mathcal E}_m\big(f\restricted{T_m},g\restricted{T_m}\big)
\end{equation}
with domain
\begin{equation}
\label{Kigami:label}
   {\mathcal F}^{\mathrm{Kigami}}
 :=
   \big\{f:T^\ast\to\mathbb{R}:\,\mbox{limit on r.h.s.\ of (\ref{Kigami:form}) exists}\big\}.
\end{equation}
Let ${\mathcal D}({\mathcal E}^{\mathrm{Kigami}})$ be the completion of ${\mathcal F}^{\mathrm{Kigami}}\cap{\mathcal C}_c(T)$ with respect to the ${\mathcal E}^{\mathrm{Kigami}}+(\boldsymbol{\cdot},\boldsymbol{\cdot})_\nu$-norm.
By Theorem~5.4 in \cite{Kigami95}, $({\mathcal E}^{\mathrm{Kigami}},\bar{{\mathcal D}}({\mathcal E}^{\mathrm{Kigami}}))$ is a regular Dirichlet form.

It was shown in Remark~3.1 in \cite{AthreyaEckhoffWinter2013} that the unique $\nu$-symmetric strong Markov process associated with $({\mathcal E}^{\mathrm{Kigami}},\bar{{\mathcal D}}({\mathcal E}^{\mathrm{Kigami}}))$
is the speed-$\nu$ motion on $(T,r).$

The bilinear form ${\mathcal E}^{\mathrm{Kigami}}$ describes the discrete time embedded Markov chains evaluated at $T_n$, $n\in\mathbb{N}$.
The fact that it is a resistance form means  that the projective limit diffusion is on ``natural scale'', which we additionally equip with speed measure $\nu$.
We can, of course, also approximate the  speed-$\nu$ motion on $(T,r)$ by continuous time  Markov chains evaluated at $T_n$, $n\in\mathbb{N}$.
Similar as in the proof of Lemma~\ref{L:discrete}, consider for each $n\in\mathbb{N}$ the map $\psi_n:T\to T_n$ which sends a point in
$T$ to the nearest point on the way from $x$ to $\rho$ which belongs to $T_n$, i.e.,
\begin{equation}
   \psi_n(x):= \sup\big\{y\in T_n:\, y\in[\rho, x]\big\},
\end{equation}
and equip $T_n$ with
\begin{equation}
   \nu_n:=\big(\psi_{n}\big)_\ast\nu.
\end{equation}
As $T^\ast$ is dense, $(\nu_n)_{n\in\mathbb{N}}$ converges vaguely to $\nu$, and thus $(T_n,r,\nu_n)_{n\in\mathbb{N}}$ converges Gromov-Hausdorff-vaguely to $(T,r,\nu)$.
It therefore follows from our invariance principle that the continuous time Markov chains which jump from $v\in T_n$ to a neighboring $v\sim v'$ at rate $(2\nu_n(\{v\})r(v,v'))^{-1}$
converges weakly in path space to the speed-$\nu$ motion on $(T,r)$.



\subsection{Invariance principle with homogeneous rescaling}
\label{Sub:Croydon}
In this subsection we relate our invariance principle to the one obtain earlier in \cite{Croydon2010}.
We first recall the excursion representation of a rooted compact measure $\mathbb{R}$-tree. We denote by
\begin{equation}
\label{e:exc}
   \mathcal E:=\big\{e\colon [0,1]\to\R_+ \bigm| e \text{ is continuous},\, e(0)=e(1)=0\big\}
\end{equation}
the set of continuous excursions on $[0,1]$. From each excursion $e\in \mathcal E$, we can define a
measure $\mathbb{R}$-tree in the following way:
\begin{itemize}
	\item $r_e(x,y):=e(x)+e(y)-2\inf_{[x,y]}e$ is a pseudo-distance on $[0,1]$,
	\item $x,y\in [0,1]$ are said to be equivalent, $x\sim_e y$, if $ r_e(x,y)=0$,
	\item the image of the projection $\pi_e\colon [0,1]\rightarrow [0,1]/{\sim_e}$ endowed with the push forward of
		$r_e$ (again denoted $r_e$), i.e.\ $T_e:=(T_e,r_e,\rho_e):=\(\pi_e([0,1]),r_e,\pi_e(0)\)$, is a
		rooted compact $\R$-tree.
	\item We endow this space with the probability measure $\mu_e:=\pi_e{}_\ast \lambda_{[0,1]}$ which is
		the push forward of the Lebesgue measure on $[0,1]$.
\end{itemize}
We denote by $g:\mathcal E\to \mathbb{T}_{c}$ the resulting ``glue function'',
\begin{equation}
\label{e:glue}
   g(e)
 :=
   \big(T_e,r_e,\rho_e,\mu_e\big),
\end{equation}
which sends an excursion to a rooted probability measure $\R$-tree.

Recall $\mathbb{T}_c$ from (\ref{e:mathbbRT}).
Given $\tree:=(T,r,\rho,\nu)\in\mathbb{T}_c$, we say that
$\tree$ satisfies a polynomial lower bound for the volume of balls, or short a {\em polynomial lower bound} if there is a $\kappa>0$ such that
\begin{equation}
\label{e:polynomial}
   \liminf_{\delta\downarrow 0} \inf_{x\in T} \delta^{-\kappa} \nu\(B_r(x,\delta)\) > 0.
\end{equation}

In \cite{Croydon2010} the following subspace of $\mathbb{T}_c$ is considered:
\begin{equation}
\label{e:Tast}
\begin{aligned}
   \mathbb{T}^\ast
   :=\big\{&\tree=(T,r,\rho,\nu)\in\mathbb{T}_c:\,
   \\
    &\mbox{ (a) $\nu$ is non-atomic, (b) $\nu$ is supported on the leaves, and}
    \\
    &\mbox{ (c) $\nu$ satisfies a polynomial lower bound.}\big\}
\end{aligned}
\end{equation}
Let $((T_n,\rho_n))_{n\in\mathbb{N}}$, be a sequence of rooted graph trees
with $\# T_n=n$, whose search-depth functions $e_n$ in ${\mathcal E}$ with uniform topology satisfy
\begin{equation}
\label{e:034}
   \tfrac{1}{a_n}e_n\tno e
\end{equation}
for a sequence $(a_n)_{n\in\mathbb{N}}$ and some $e\in{\mathcal E}$ with $(T_e,r_e,0,\mu_e)\in \mathbb{T}^\ast$.
In Theorem~1.1 of \cite{Croydon2010}, it is shown that  the discrete-time simple random walks on $T_n$ starting in $\rho_n$ with jump sizes rescaled by $1/a_n$ and
speeded up by a factor of\/ $n\cdot a_n$ converge to the $\mu_e$\nbd Brownian motion on $T_e$ starting in $0$.

To connect the above construction with Theorem~\ref{T:001} notice that the map
$g$ from (\ref{e:glue}) is continuous if $\mathbb{T}_{c}$ is endowed with the rooted Gromov-Hausdorff-weak topology, and $\mathcal E$ with the uniform topology (see
\cite[Proposition~2.9]{AbrahamDelmasHoscheit:exittimes}; compare also \cite[Theorem~4.8]{Loehr} for a generalization to lower semi-continuous excursions).
Thus it follows from (\ref{e:034}) that if we put
   $\nu_n:=\mu_{a_n^{-1} e_n}$, 
then $(T_n,\nu_n)$ converges to $(T_e,\mu_e)$ rooted Gromov-Hausdorff-weakly. Analogously to
Example~\ref{Exp:002} we obtain that
$d_{\mathrm{Pr}}^{(T_n,r_n)}(\nu_n,\tilde{\nu}_n)\le a_n^{-1}$, where
\begin{equation}
\label{e:035}
   \tilde{\nu}_n(\{v\}):=\tfrac{\deg(v)}{2n},
\end{equation}
and that thus also  $(T_n,\tilde{\nu}_n)$ converges to $(T_e,\mu_e)$ rooted Gromov-Hausdorff-weakly by
\cite[Lemma~2.10]{ALW2}.
Theorem~\ref{T:001} then implies that unit rate simple random walks with edge lengths rescaled by $a_n^{-1}$ and
speeded up by $n\cdot a_n$ converge to the speed-$\mu_e$ motion on $(T_e,r_e)$.  As $\mu_e$ always has full
support, the requirement that $\mu_e$ is supported on the leaves already implies that $(T_e,r_e)$ is an
$\mathbb{R}$-tree and thus the speed-$\mu_e$ motion on $(T_e,r_e)$ has continuous paths.

Note that in contrast to \cite{Croydon2010} our Theorem \ref{T:001} does not require any additional assumptions
on the limiting tree, which also does not have to be an $\R$-tree. The polynomial lower bound or that $\nu$
is non-atomic and supported on the leaves are not required.
Also note that Theorem~1.1 of \cite{Croydon2010} does only allow for homogeneous (non-state-dependent) rescaling.
This means, for example, that in the particular case where the trees $(T_n, r_n)$ are subsets of $\R$, only the
case $T_n=a_n^{-1}\Z \cap [0,na_n^{-1}]$ and $\nu_n\(\{x\}\)=n^{-1}$, $x\in T_n$, is covered.


\subsection{Random walk on the size-biased branching tree}
\label{Sub:Kesten}
Theorem \ref{T:001} applies to trees that are complete and locally compact. The extension from compact  to complete, locally compact trees is relatively straight forward.
However this extension helps us to cover the random walk on the size-biased Galton-Watson tree studied in
\cite{MR88b:60232} in the annealed regime and in \cite{BarlowKumagai2006} in the quenched regime.
In this subsection we want to illuminate these results and put them in the context of our invariance principle.

Consider a random graph theoretical tree ${\mathcal T}_{\mbox{\tiny Kesten}}$ which is distributed like the
rooted Galton-Watson process with finite variance mean~$1$ offspring distribution conditioned to never die out.
Let $X$ be the (discrete-time) nearest neighbor random walk on ${\mathcal T}_{\mbox{\tiny Kesten}}$ and
$d$ the graph distance on ${\mathcal T}_{\mbox{\tiny Kesten}}$. Consider the rescaled height process
\begin{equation}
\label{e:height}
   Z_t^{(n)}:=n^{-\frac{1}{3}}\cdot d\big(\rho,X_{\lfloor nt\rfloor}\big), \qquad t\ge 0.
\end{equation}

In \cite{MR88b:60232} it is shown that if
 $  \tau_{B^c(\rho,N)}
 :=
   \inf\big\{n\ge 0:\,d(\rho,X_n)=N\big\}$,
then for all $\varepsilon>0$ there exists $\lambda_1,\lambda_2$ such that
under the annealed law~$\mathbb{P}^\ast$,
   $$\mathbb{P}^\ast\big\{\lambda_1\le N^{-3}\tau_{B^c(\rho,N)}\le\lambda_2\big\}\ge 1-\varepsilon,$$
for all $N\ge 1$. Moreover, under $\mathbb{P}^\ast$, the process $Z^{(n)}$ converges weakly in path space to a non-trivial process $Z$ with continuous paths.

In contrast to this annealed regime, in \cite{BarlowKumagai2006}  (in the continuous time
setting) it is shown that for almost all realizations of\/ ${\mathcal T}_{\mbox{\tiny Kesten}}$,
the family $\{Z^{(n)};\,n\in\mathbb{N}\}$ is not tight.

These two statements relate to our invariance principle as follows. Recall from (\ref{e:exc}) the space of
continuous excursions on $[0,1]$ and from (\ref{e:glue}) the glue map $g$ which sends an excursion $e\in{\mathcal E}$
to a rooted metric tree $([0,1]/{\sim_e},r_e,0)$ as well the map $\pi_e$ which, given $e\in{\mathcal E}$,
sends a point from the excursion interval  $[0,1]$ to $T_e$. We can easily extend the maps $g$ and $\pi_e$ to the space
\begin{equation}
\label{e:expinfty}
   {\mathcal E}_\infty:=\bigl\{e\colon \R\to\R_+ \bigm| e \text{ is continuous},\, e(0)=0,\,
   \lim_{x\to\pm\infty}e(x)=\infty\bigr\}
\end{equation}
of continuous, two-sided, transient excursions on $\R$. To this end, we use the semimetric defined by
\begin{equation}
	r_e(x,y) := \begin{cases}
		e(x)+e(y) - 2\inf_{z\in [x,y]} e(z), & xy \ge 0,\\
		e(x)+e(y) - 2\inf_{z\in \R\setminus[x,y]} e(z), & xy < 0
		\end{cases}
\end{equation}
for $x\le y$ (see \cite{Duquesne09}).
Then $g(e)$ is a rooted locally compact metric measure tree with a
boundedly finite measure, for all $e\in {\mathcal E}_\infty$.
It is not hard to show that the map $g$ from (\ref{e:glue}) is continuous if $\mathbb{T}$ is endowed with the
rooted Gromov-Hausdorff-vague topology, and $\mathcal E_\infty$ with the uniform topology on compact sets (see \cite[Proposition~7.5]{ALW2}).

In the particular case of a geometric offspring distribution, ${\mathcal T}_{\mbox{\tiny Kesten}}$ can be
associated with the (two-sided) random excursion $\tilde{W}$, where for all $t \in \R$,
\begin{equation}
   \tilde{W}_t:=\begin{cases} W_t-2\inf_{s\in[0,t]}W_s, & t\ge 0\\
   				W_t-2\inf_{s\in[t,0]}W_s, & t < 0,
		\end{cases}
\end{equation}
with a simple two-sided random walk path $(W_n)_{n\in\mathbb{Z}}$, $W_0=0$, linearly interpolated.
As $W$ converges, after Brownian rescaling, weakly in path space towards (two-sided) standard Brownian motion
$(B_t)_{t\in \R}$, we have
\begin{equation}
\label{e:WtoB}
   \big(n^{-1/3}\tilde{W}_{n^{2/3}t}\big)_{t\in\R}\,\Tno\big(\tilde{B}_t\big)_{t\in \R},
\end{equation}
where $\tilde{B}_t:=B_t-2\inf_{s\in[0\land t,\,t \lor 0]}B_s$.

Given a realization $e$ of $\tilde{W}$, define $e_n:=n^{-1/3}e(n^{2/3}\bdot)\in\E_\infty$ and denote by $\nu_n$
the rescaled degree measure on $T_{e_n}$, i.e., for all $A\subseteq T_{e_n}$,
\begin{equation}
	\nu_n(A):=n^{-2/3}\sum_{v\in A}\tfrac{1}{2}\deg(v).
\end{equation}
By Proposition~2.8 in \cite{BarlowKumagai2006},  for almost all realizations $e$ of $\tilde{W}$,
\begin{equation}
\label{e:non-convergence}
   \liminf_{n\to\infty}\nu_n\big(B(\rho,R)\big)=0,\quad\text{and}\quad\limsup_{n\to\infty}\nu_n\big(B(\rho,R)\big)=\infty,
\end{equation}
and thus the sequence $\{\nu_n;\,n\in\mathbb{N}\}$ does not converge. Consider once more the map which sends all
points of a half edge to its end point, and notice that the image measure of
$\mu_{e_n}=(\pi_{e_n})_\ast\lambda_{\R_+}$ under this map equals $\nu_n$. Thus the Prohorov distance between
$\mu_{e_n}$ and $\nu_n$ is at most $n^{-1/3}$, and thus for almost all realizations $e$ of $\tilde{W}$,
 also the sequence $\{\mu_{e_n};\,n\in\mathbb{N}\}$ does not converge. Hence the assumptions on our invariance
 principle fail for almost all realizations of ${\mathcal T}_{\mbox{\tiny Kesten}}$.

Notice that we can choose for each $n\in\mathbb{N}$ a realization $e_n$ of
$n^{-1/3}\tilde{W}_{n^{2/3}\boldsymbol{\cdot}}$, and a realization $e$ of $\tilde{B}$, such that $e_n\tno e$,
almost surely. To understand why the quenched rescaling failed, notice that  $e_n\tno e$  {\sc cannot} be
realized via a coupling such that all the $e_n$ come from the same realization of $\tilde{W}$.
As now $g(e_n)$ clearly converges to $g(e)$ by continuity of $g$, Theorem~\ref{T:001} implies that the
speed-$\mu_{e_n}$ random walk $X^n$ on $(T_{e_n},r_{e_n})$ starting in $\rho_{e_n}$ converges weakly in path space to
the $\mu_e$\nbd Brownian motion $X=(X_t)_{t\ge 0}$ on $(T_e,r_e)$ started in $\rho_e$ for almost all
realizations.
We can interpret this as \emph{annealed convergence} in law of $X^n$ to $X$, which we define -- in analogy to
\defref{003} and in view of Skorohod's representation theorem -- as follows.
There exists a coupling of the underlying random spaces $\tree=(T_e,r_e,\mu_e)$, $\treen = (T_{e_n}, r_{e_n},
\mu_{e_n})$, $n\in\N$, such that almost surely, conditioned on these spaces, $X^n$ converges weakly in path
space to $X$ in the sense of \defref{003}.
In particular, the rescaled height processes $Z^{(n)}$, defined in \eqref{e:height}, converge
under the annealed law to the height process $Z=(Z_t)_{t\ge0}$ defined by $Z_t:=r_e(\rho_e,X_t)$.
As $X$ is recurrent by Theorem~4 in \cite{AthreyaEckhoffWinter2013}, its life time is infinite, and $Z$ is
non-trivial.

\subsection{Motions on $\Lambda$-coalescent measure trees}
\label{Sub:Lambda}
We conclude the example section with the example of speed-$\nu$ motions on the $\Lambda$-coalescent measure trees for appropriate measures $\nu$.
These have not been considered in the literature so far.

Let $\Lambda$ be a finite measure on $([0,1],{\mathcal B}([0,1]))$ which satisfies
\begin{equation}
\label{e:coming}
   \sum_{n=2}^\infty\Bigl(\int^1_0\sum_{k=2}^n{n\choose k}(k-1)  x^{k-2}(1-x)^{n-k} \Lambda(\mathrm{d}x)\Bigr)^{-1}<\infty.
\end{equation}
Denote by $\mathbb{S}$ the set of all partitions of $\mathbb{N}$, and for each $n\in\mathbb{N}$ by $\mathbb{S}_n$ the set of all partitions of $\{1,...,n\}$. Write $\rho_n$ for the restriction map from $\mathbb{S}$ to $\mathbb{S}_n$.

The $\Lambda$-coalescent is the unique $\mathbb{S}$-valued strong Markov process $\zeta$, such that for each
$n\in\mathbb{N}$ the restricted process $\rho_n(\zeta)$ is the following $\mathbb{S}_n$-valued continuous time
Markov chain. Given the current partition ${\mathcal P}\in\mathbb{S}_n$, every $k$-tuple of its partition
elements merges independently at rate
\begin{equation}\label{eq:lambdarate}
	\lambda_{k,\#{\mathcal P}}:=\int\Lambda(\mathrm{d}x)\,x^{k-2}(1-x)^{\#{\mathcal P}-k}
\end{equation}
into one partition element, thereby forming a new partition. It is known that condition (\ref{e:coming}) is
equivalent to the $\Lambda$-coalescent coming down from infinity, i.e., under (\ref{e:coming}),
$\#\zeta_t<\infty$ for each $t>0$, almost surely (\cite{Schweinsberg2000}). Furthermore, (\ref{e:coming})
implies the so-called {\em dust-free property}, i.e., $\int_0^1\Lambda(\mathrm{d}x)\,x^{-1}=\infty$.

Equip for each realization of the $\Lambda$-coalescent started in ${\mathcal P}_0:=\{\{i\}:\,i\in\mathbb{N}\}$
the set $\mathbb{N}$ with the genealogical distances, i.e., $r(i,j)$ is for all $i,j\in\mathbb{N}$ the first time when
$i$ and $j$ belong to the same partition element. Denote the completion of $(\N,r)$ by $({\mathcal
T}_{\mbox{\tiny$\Lambda$}},r)$. Obviously, coming down from infinity implies (and is in fact equivalent to) the compactness of ${\mathcal T}_{\mbox{\tiny$\Lambda$}}$. Further, equip for each $n\in\mathbb{N}$, ${\mathcal T}_{\mbox{\tiny$\Lambda$}}$
with the sampling measure $\mu^n:=\tfrac{1}{n}\sum_{i=1}^n\delta_i$. By Theorem~4 in \cite{GrevenPfaffelhuberWinter2009} the sequence
$(({\mathcal T}_{\mbox{\tiny$\Lambda$}},r,\mu^n))_{n\in\mathbb{N}}$ converges weakly in Gromov-weak topology towards the so-called {\em $\Lambda$-coalescent measure tree},
$({\mathcal T}_{\mbox{\tiny$\Lambda$}},r,\mu)$.

Consider next the $\R$-tree $(\bar{{\mathcal T}}_{\mbox{\tiny$\Lambda$}}, \bar{r})$ spanned by
$({\mathcal T}_{\mbox{\tiny$\Lambda$}},r)$, and notice that ${\mathcal T}_{\mbox{\tiny$\Lambda$}}$ is ultra-metric.
We therefore find a unique point $\rho\in \bar{{\mathcal T}}_{\mbox{\tiny$\Lambda$}}$ whose distance to ${\mathcal T}_{\mbox{\tiny$\Lambda$}}$ equals $\mathrm{diam}(\bar{{\mathcal T}}_{\mbox{\tiny$\Lambda$}})/2$, which we choose as the root.
For each point $x\in \bar{{\mathcal T}}_{\mbox{\tiny$\Lambda$}}$ denote by
\begin{equation}
\label{e:subtree}
   S^x:=\big\{z\in {\mathcal T}_{\mbox{\tiny$\Lambda$}}:\,x\in[\rho,z]\big\}
\end{equation}
the (leaves of the) subtree above $x$, and recall from (\ref{length}) the notion of the length measure
$\lambda^{(T,r,\rho)}$ of a rooted compact metric tree $(T,r,\rho)$.

Define the speed measures $\nu^n$, $n\in\mathbb{N}$, and $\nu$ on $\bar{{\mathcal T}}_{\mbox{\tiny$\Lambda$}}$ as being absolutely continuous with respect to the length measure with densities
\begin{equation}
\label{e:denss}
   \tfrac{\mathrm{d}\nu^n}{\mathrm{d}\lambda^{\bar{{\mathcal T}}_{\mbox{\tiny$\Lambda$}}}}(x)
   :=\mu^n\big(S^x\big),\;\mbox{ and }\;\tfrac{\mathrm{d}\nu}{\mathrm{d}\lambda^{\bar{{\mathcal T}}_{\mbox{\tiny$\Lambda$}}}}(x) :=\mu\big(S^x\big).
\end{equation}
for all $x\in \bar{{\mathcal T}}_{\mbox{\tiny$\Lambda$}}$.
 Obviously, $\nu^n$, $n\in\mathbb{N}$, and $\nu$ are finite measures with total masses at most (and in fact due
 to the dust-free property equal to) $\mathrm{diam}(\bar{{\mathcal T}}_{\mbox{\tiny$\Lambda$}})/2$. Note that  for every
 ultrametric space $(T,r)$, the map $\xi^{(T,r)}$ which sends a pair $(t,x)\in[0,\infty)\times T$ to the unique ``ancestor'' of $x$ a time $t$ back, i.e., the unique
$y\in \bar{T}$  ($\bar{T}$ denoting the span of $T$) with $\bar{r}(y,x)=t\wedge\tfrac{1}{2}\mathrm{diam}(\bar{T})$ is continuous. Hence using the convergence
alluded to earlier (Theorem~4 in \cite{GrevenPfaffelhuberWinter2009}) the sequence $((\bar{{\mathcal T}}_{\mbox{\tiny$\Lambda$}},\nu^n))_{n\in\mathbb{N}}$ converges weakly in Gromov-weak topology towards
$(\bar{{\mathcal T}}_{\mbox{\tiny$\Lambda$}},\nu)$. Our invariance principle therefore implies that the
$\nu_n$\nbd Brownian motion on $(\mathrm{supp}(\nu^n),\bar{r})$ converges weakly  to the $\nu$\nbd Brownian
motion on $(\TLambdacl,\bar{r})$ in the sense of finite dimensional marginals (provided all Brownian motions start at the same point). Applying once more the dust-free property implies that the global lower mass-bound
holds, and thus the convergence holds even in path space.

We can modify the example such that we obtain path-wise convergence of a continuous time Markov chain to a motion on a totally disconnected (limiting) tree.
For that purpose, denote by $\mathrm{Br}(\bar{{\mathcal T}}_{\mbox{\tiny$\Lambda$}})$ the set of branch points of
$\bar{{\mathcal T}}_{\mbox{\tiny$\Lambda$}}$, i.e., the set of those $x\in\bar{{\mathcal
T}}_{\mbox{\tiny$\Lambda$}}$ such that either $x=\rho$ or $\bar{{\mathcal
T}}_{\mbox{\tiny$\Lambda$}}\setminus\{x\}$ consists of at least $3$ connected components.
Consider now the (atomic) length measure on $\BrTLambda$ and the Dirac measure $\delta_\rho$, and define
\begin{equation}\label{eq:length+rho}
	\hat\lambda := \lambda^{(\mathrm{Br}(\bar{{\mathcal T}}_{\mbox{\tiny$\Lambda$}}),\bar{r},\rho)} + \delta_\rho.
\end{equation}
We use the speed measures $\tilde{\nu}^n$, $n\in\mathbb{N}$, and $\tilde{\nu}$ on $\bar{{\mathcal
T}}_{\mbox{\tiny$\Lambda$}}$ which are absolutely continuous with respect to $\hat\lambda$ with densities
\begin{equation}
\label{e:denss2}
   \tfrac{\mathrm{d}\tilde{\nu}^n}{\mathrm{d}\hat\lambda}(x)
   :=\mu^n\big(S^x\big),\;\mbox{ and }\;\tfrac{\mathrm{d}\tilde{\nu}}{\mathrm{d}\hat\lambda}(x)
   :=\mu\big(S^x\big)
\end{equation}
for all $x\in \mathrm{Br}(\bar{{\mathcal T}}_{\mbox{\tiny$\Lambda$}})$. For each
$\varepsilon\in(0,\tfrac{1}{2}\mathrm{diam}(\bar{{\mathcal T}}_{\mbox{\tiny$\Lambda$}})))$ and for all suitably
large $n\in\mathbb{N}$, we have
$\mathrm{supp}(\tilde{\nu}^n)\cap\{x\in \BrTLambda:\,\bar{r}(x,{\mathcal T}_{\mbox{\tiny$\Lambda$}})\ge\varepsilon\} =
 \{x\in \BrTLambda:\,\bar{r}(x,{\mathcal T}_{\mbox{\tiny$\Lambda$}})\ge\varepsilon\}$.
Therefore, the sequence $((\bar{{\mathcal T}}_{\mbox{\tiny$\Lambda$}},\tilde{\nu}^n))_{n\in\mathbb{N}}$ also
converges weakly in Gromov-weak topology towards $(\bar{{\mathcal T}}_{\mbox{\tiny$\Lambda$}},\tilde{\nu})$.
Thus our invariance principle applies to the speed-$\tilde{\nu}^n$ random walk on $\mathrm{supp}(\tilde{\nu}^n)$
and the speed-$\tilde{\nu}$ motion on $\mathrm{supp}(\tilde{\nu})=\BrTLambda \cup \TLambda$.

\begin{acknowledgements}
	We would like to thank Steve Evans and Fabian Gerle for discussions, and the anonymous referees for
	their detailed reports, which enabled us to improve the paper.
\end{acknowledgements}

\bibliographystyle{alpha}
\bibliography{lit}

\end{document}